\documentclass[11pt,a4paper,reqno]{amsart} 
\usepackage[english]{babel}
\usepackage[utf8]{inputenc}
\usepackage{graphicx,amssymb,latexsym,amsmath,amscd,amsthm}
\usepackage{verbatim,hyperref, color}
\usepackage{cite} 
\usepackage[margin=1.3in]{geometry} 
\newtheorem{lemma}{Lemma}[section]

\usepackage{slashbox}
\newtheorem{prop}[lemma]{Proposition} 
\newtheorem{thm}[lemma]{Theorem}
\newtheorem{cor}[lemma]{Corollary} 
\theoremstyle{definition} 
\newtheorem{Def}[lemma]{Definition}
  \theoremstyle{remark} 
\newtheorem{rem}[lemma]{Remark} 
\newtheorem{rems}[lemma]{Remarks}  
\newcommand{\F}{\mathcal{F}}
\newcommand{\G}{\mathcal{G}}

\newcommand{\D}{\mathcal{D}}
  
\newcommand{\N}{\mathbb{N}}
\newcommand{\M}{\mathcal{M}}

\newtheorem{innercustomthm}{Theorem}
\newenvironment{customthm}[1]
{\renewcommand\theinnercustomthm{#1}\innercustomthm}
{\endinnercustomthm}
\title{Subshifts of Finite Type with a Hole} 
\author[Haritha C]{Haritha Cheriyath}
\address{Department of Mathematics\\
Indian Institute of Science Education and Research Bhopal\\
Bhopal Bypass Road, Bhauri \\
Bhopal 462 066, Madhya Pradesh\\
India}
\email{harithacheriyath@gmail.com}
\author[Nikita Agarwal]{Nikita Agarwal}
\address{Department of Mathematics\\
Indian Institute of Science Education and Research Bhopal\\
Bhopal Bypass Road, Bhauri \\
Bhopal 462 066, Madhya Pradesh\\
India}
\email{nagarwal@iiserb.ac.in}

\begin{document} 
\maketitle
\begin{abstract} 
	We consider a subshift of finite type on $q$ symbols with a union of $t$ cylinders based at words of identical length $p$ as the hole. We explore the relationship between escape rate into the hole and a rational function $r(z)$, of correlations between forbidden words in the subshift with the hole. In particular, we prove that there exists a constant $D(t,p)$ such that if $q>D(t,p)$, then escape rate is faster into the hole when the value of the corresponding rational function $r(z)$ evaluated at $D(t,p)$ is larger. Further, we consider holes which are unions of cylinders based at words of identical length, having zero cross-correlations, and prove that the escape rate is faster into the hole with larger Poincar\'e recurrence time. Our results are more general than the existing ones known for maps conjugate to a full shift with a single cylinder as the hole. 
\end{abstract}

\section{Introduction} \label{sec:intro}
Open dynamical systems or dynamical systems with a hole are of interest because of their dynamical properties and also their applications, we refer to~\cite{BB, BY,DemYoung,FS,Combinatorial}. In~\cite{PY}, Pianigiani and Yorke introduced such systems and also discussed the issue of the rate of the escape of orbits into the hole (known as the \textit{escape rate}). A natural question in this context is whether and how the escape rate depends on the position and the size of the hole. This question was first addressed by Bunimovich and Yurchenko~\cite{BY}. Their study focuses on full shifts on finitely many symbols with a single cylinder as a hole. They prove that the escape is the fastest through the hole whose minimal period is maximum. Their results are applicable and restricted to all maps conjugate to a full shift with a single cylinder as the hole. Similar results are not known to exist even for simple systems such as expansive Markov maps. Froyland and Stancevic~\cite{FS} and Haritha and Agarwal~\cite{Product} present exploratory numerical examples for other shift spaces such as subshifts of finite types. However, these papers do not provide general results.

When the underlying state space is a subshift of finite type, the problem of computing the escape rate into a Markov hole, which is a finite union of cylinders, reduces to computing $f(n)$, the number of words of length $n$, which do not contain any of the words from a given finite collection of words with symbols from a fixed set of symbols, see~\cite{BY,Product}. The problem of computing $f(n)$ has been extensively studied in the literature. This problem is interesting in its own right and has applications to probability theory and combinatorics, see~\cite{Combinatorial,HXYC,String}. There is no explicit formula for $f(n)$, thus often, one studies its generating function $F(z)$. Guibas and Odlyzko in~\cite{Guibas,Combinatorial} gave an explicit combinatorial method (see Theorem~\ref{thm:formF}) to compute the generating function $F(z)$, using a system of linear equations involving $q$, the size of the symbol set, and $r(z)$, a rational function of the correlation polynomials between the words in the collection. The correlation polynomials capture overlapping between the words. The function $F(z)$ is rational and its special form is helpful to understand the asymptotic behavior of $f(n)$. This explicit form of the generating function $F(z)$ was used in~\cite{BY} to study the escape rate of the shift map on the full shift on two symbols into a Markov hole (a single cylinder).  

In this paper, we consider a general class of maps, those conjugate to an irreducible subshift of finite type, with finite union of cylinders based at words of identical length as holes. We extend and generalize the results of~\cite{BY,Product} for this set-up. The results in these works follow as a special case of the results in our paper. Unlike the case of~\cite{BY}, when the hole corresponds to a union of cylinders, that is, when there is more than one forbidden word, the cross-correlations between each pair of forbidden words appear in the generating function, which makes the analysis harder. For a full shift on $q$ symbols with a union of $t$ cylinders based at words each of length $p$ as the hole, we prove that there is a constant $D(t,p)$ such that for all $q>D(t,p)$, the escape rate is faster into the hole when the value of the corresponding rational function $r(z)$ evaluated at $D(t,p)$ is larger. In particular, when the hole is a union of two cylinders, $D(t,p)=3p^2+2$. These results extend smoothly to subshifts of finite type with a finite union of cylinders as a hole.  

We also explore the relationship between the escape rate into the hole and the minimal period of the hole (which is the same as the Poincar\'e recurrence time if the cross-correlations between words at which the cylinders describing the hole are based are zero; see Remark~\ref{rem:prt_mp}). In Theorems~\ref{thm:esc_min1} and ~\ref{thm:gen-period}, we consider two collections with the same number of words, each of equal length, having zero cross-correlations, and prove that the collection with larger minimal period has the larger escape rate. We give counter-examples in Remark~\ref{rem:cross-nonzero} discussing the violation of this result when the cross-correlations are non-zero. 

\subsection{Organization of the paper} 
In Section~\ref{sec:prelim}, we present some preliminaries on subshift of finite type, the relationship between the escape rate into a hole (which is a union of cylinders) and the topological entropy, and the form of the corresponding generating function in terms of the size of the symbol set $q$ and the associated correlation polynomials. We state the main results of this paper in Section~\ref{sec:main-results}, the proofs of which are presented in later sections. In Section~\ref{sec:full_two}, we consider a full shift as the underlying space and compare the escape rate into two holes which are unions of two cylinders based at words of equal length $p$. Theorems~\ref{thm:esc_r_1} and~\ref{thm:esc_r_2} give the relationship between the escape rate and the corresponding rational function $r(z)$ with certain assumptions on $p$ and $q$. Further, in Theorem~\ref{thm:esc_min1}, we discuss the relationship between the minimal period of the hole and the escape rate. In Section~\ref{sec:full_morethantwo}, we generalize the results obtained in Section~\ref{sec:full_two} when holes are unions of more than two cylinders. Theorems~\ref{thm:gen} and~\ref{thm:gen-period} are generalizations of Theorems~\ref{thm:esc_r_2} and~\ref{thm:esc_min1}, respectively. In Section~\ref{sec:subshift}, we discuss applications of the results obtained in Sections~\ref{sec:full_two} and~\ref{sec:full_morethantwo} to the case when the underlying space is an irreducible subshift of finite type. We give concluding remarks in Section~\ref{sec:conc}. 

\section{Preliminaries} \label{sec:prelim}
In this section, we present the concepts, notations and definitions that we use throughout the paper.

\subsection{Subshifts of finite type}
Let $\Sigma=\{0,1,\ldots,q-1\}$ be a set of symbols with $q\geq 2$. We denote by $\Sigma^\mathbb{N}$, the set of all one-sided sequences with symbols from $\Sigma$. It is known as a \textit{one-sided full shift on $q$ symbols}. Let $\F$ be a finite collection of words with symbols from $\Sigma$. The collection of all sequences from $\Sigma^\mathbb{N}$ which do not contain any of the words from the collection $\F$ as subwords is denoted by $\Sigma_{\F}$. This is known as a \textit{$(p-1)$-step subshift} of finite type if the longest word in $\F$ has length $p\ge 1$. Without loss of generality, it can be assumed that each word in the collection $\F$ has length $p$. The words in the collection $\F$ are called \textit{forbidden}. An \textit{allowed word} in $\Sigma_\F$ is a word which appears as a subword in some sequence in $\Sigma_\F$. Let $\sigma:\Sigma_\F\to\Sigma_\F$ be the \textit{left shift map}.

A zero-step subshift is a full shift on smaller set of symbols. Every $(p-1)$-step subshift is conjugate to a one-step subshift via a block conjugacy map, see~\cite{LM_book}.

An adjacency matrix of a one-step shift on $q$ symbols is a binary matrix $A$ of size $q$ where the $ij^{th}$ entry is 1 if and only if $ij$ is an allowed word. If $A$ is an irreducible matrix, that is, for every $i,j$, there exists $n\ge 1$ such that $A^n_{i,j}>0$, we say that the subshift is \textit{irreducible}. It is said to be \textit{reducible} if it is not irreducible. The subshift is said to be \textit{primitive} if its adjacency matrix $A$ is primitive, that is, there exists $n\ge 1$ such that $A^n>0$. 

A $(p-1)$-step subshift $X$ is said to be \textit{irreducible (primitive)} if the one-step subshift $Y$ to which it is conjugate to is irreducible (primitive). Let $A$ be the adjacency matrix of $Y$. For convenience, the adjacency matrix of $X$ is also defined as $A$.

Let $\Sigma_\F$ be an irreducible subshift of finite type. If $f(n)$ denotes the number of words of length $n$ which appear as subwords in sequences in $\Sigma_{\F}$, then the \textit{topological entropy} of the subshift is given by
$	h_{\text{top}}(\Sigma_{\F}) = \lim_{n \to \infty}\ln\left(f(n)^{1/n}\right)=
\ln(\theta),$
where $\theta$ is the largest (in modulus) eigenvalue of its adjacency matrix, which is real and positive by the Perron-Frobenius theorem, and is known as the \textit{Perron root}.

\subsection{The Parry measure}
William Parry showed the existence and uniqueness of a $\sigma$-invariant measure of maximal entropy for irreducible subshifts of finite type, see~\cite{Parry64}. This measure is now known as the Parry measure, which we now describe.

Consider an irreducible one-step subshift $\Sigma_\F$ on $q$ symbols.  Let $w=i_1i_2\ldots i_n$ be an allowed word in $\Sigma_{\F}$, and let 
\[
C_w=\{x_1x_2\ldots\in\Sigma_{\F}\ \vert\ x_1=i_1,x_2=i_2,\ldots, x_n=i_n\},
\]
be the \textit{cylinder based at the word $w$}, which is the collection of all sequences in $\Sigma_\F$ which begin with the word $w$. Then we obtain a probability measure space with set $\Sigma_{\F}$, $\sigma$-algebra generated by the cylinders based at all allowed words in $\Sigma_\F$ of finite length, and the measure $\mu$ defined as follows: for every allowed word $w=i_1i_2\ldots i_n$, 
\[
\mu(C_w)=\dfrac{u_{i_1}v_{i_n}}{\theta^{n-1}},
\]
where $\theta$ is the Perron root of the adjacency matrix of $\Sigma_\F$, $v = (v_0,\ldots,v_{q-1})^T$, and $u = (u_0,\ldots,u_{q-1})$,
are the normalized right and left eigenvectors with respect to $\theta$ such that $u^Tv=1$. The measure $\mu$ is known as the \textit{Parry measure}. This probability measure is invariant and ergodic with respect to the shift map $\sigma$ on $\Sigma_\F$ and is the unique measure that maximizes the measure theoretic entropy according to the variational principle. Note that there is a large class of $\sigma$-invariant measures, known as Markov measures, on subshifts of finite type. The Parry measure is a special Markov measure, see~\cite{Parry64}.

In the case of full shift ($\F=\emptyset$), $\mu(C_w)=1/q^n$, for all words $w$ of length $n$. If $\F\ne\emptyset$, it is immediate from the definition of $\mu$ that two cylinders based at words of identical length need not have the same measure.
On any subshift of finite type, we define the Parry measure using the conjugacy between the subshift and the one-step shift. 

\subsection{Escape rate}

Consider an irreducible subshift of finite type $\Sigma_\F$ with Parry measure $\mu$. Let $\G$ be a finite collection of allowed words in $\Sigma_\F$. Let $H_\G=\bigcup_{w\in\G}C_w$, referred to as a \textit{hole}, denote the union of cylinders in $\Sigma_\F$ based at words from $\G$. The \textit{escape rate of the shift map $\sigma\vert_{\Sigma_\F}$ into the hole $H_\G$} is defined as
\[
\rho(\G;\Sigma_{\F}) := -\lim_{n\rightarrow \infty} \dfrac{1}{n} \ln\mu\left(\mathcal{W}_n(\G)\right),
\]
where $\mathcal{W}_n(\G)$ denotes the collection of all sequences in $\Sigma_{\F}$ which do not contain words from $\G$ as subwords in their first $n$ positions. In the case of a full shift, that is, when $\F=\emptyset$, we use the shorthand $\rho(\G)$ for $\rho(\G;\Sigma_{\F})$. 

The notion of escape rate is quite general but we deal with subshifts of finite type in this paper, hence we restrict the definition to this setting only. Also the escape rate may not exist in general but in the setting described in the above definition, it does exist. It can be expressed in terms of the topological entropies of the subshifts $\Sigma_\F$ and $\Sigma_{\F\cup\G}$.

\begin{thm}~\cite[Theorem 3.1]{Product}\label{thm:esc_rate}
	The escape rate $\rho(\G;\Sigma_{\F})$ can be expressed as
	\[
	\rho(\G;\Sigma_{\F})= h_{\text{top}}(\Sigma_{\F}) - h_{\text{top}}(\Sigma_{\F\cup \G}).
	\]
\end{thm}

By this result, the problem of computing the escape rate turns into a purely combinatorial problem. Moreover due to this result, comparison of escape rates into two holes $H_{\G_1}$ and $H_{\G_2}$ in a given subshift of finite type $\Sigma_\F$ with $q$ symbols reduces to the comparison of the topological entropies of the subshifts $\Sigma_{\F\cup \G_1}$ and $\Sigma_{\F\cup \G_2}$. This is further equivalent to comparing the escape rates into two holes $H_{\F\cup \G_1}$ and $H_{\F\cup \G_2}$ in the full shift on $q$ symbols. Thus we consider underlying space to be a full shift till Section~\ref{sec:full_morethantwo}. In Section~\ref{sec:subshift}, we consider subshifts of finite type.

\subsection{Combinatorics on words}

We now review some concepts from combinatorics on words which are used in due course (see~\cite{Combinatorial}).

\begin{Def}
	Let $u$ and $w$ be two words of lengths $p_1$ and $p_2$, respectively with symbols from $\Sigma$. The \textit{correlation polynomial} of $u$ and $w$ is defined as
	\[
	(u,w)_z = \sum_{\ell=1}^{p_1} b_\ell z^{p_1-\ell}, 
	\]
	where $b_{\ell}=0$, if $\sigma^{\ell-1}(C_u)\cap C_w=\emptyset$, and $b_{\ell}=1$, otherwise, with $C_u$ and $C_w$ denote the cylinders on $\Sigma^\N$ based at $u$ and $w$, respectively. \\
	When $u=w$, $(u,u)_z$ is said to be the \textit{autocorrelation polynomial} of $u$, and when $u\neq w$, $(u,w)_z$ is said to be the \textit{cross-correlation polynomial} of $u$ and $w$. 
\end{Def}
Note that if $\sigma^{\ell-1}(C_u)\cap C_w=\emptyset$, for all $1\leq \ell\leq p_1$, then $(u,w)_z=0$. Also, if $u\neq w$, $b_1=0$ and if $u=w$, then $b_1=1$. Hence for $u\neq w$, $(u,w)_z$ is a polynomial with degree at most $p_1-2$, and $(u,u)_z$ is a monic polynomial with degree exactly equal to $p_1-1$. 

A collection $\{w_1,w_2,\ldots,w_t\}$ of words is said to be \textit{reduced} if for any $i\ne j$, $w_i$ is not a subword of $w_j$. 

Let $\F=\{w_1,w_2,\ldots,w_t\}$ be a reduced collection of words with symbols from $\Sigma$. Let $f(n)$ denote the number of allowed words in $\Sigma_\F$ and
\begin{eqnarray}\label{eq:gen}
	F(z)&:=&\sum_{n=0}^\infty f(n)z^{-n}
\end{eqnarray}
be the generating function for the sequence $(f(n))_{n\ge 0}$. The function $F(z)$ is rational and its form is described in the following result.

\begin{thm}~\cite{Guibas}\label{thm:formF}
	With the notations as above,
	\begin{equation*}
		F(z)=\dfrac{z}{z-q+r_\F(z)},
	\end{equation*}
	where $r_\F(z)$ is the sum of entries of the inverse of the correlation matrix function $\mathcal{M}(z)=[(w_j,w_i)_z]_{i,j}$. 
\end{thm}

\begin{rems}\label{rem:formF}
	1) If the length of each word in $\F$ is $p$, then the numerator of $r_F(z)$ is a polynomial of degree $(t-1)(p-1)$ and the denominator of $r_F(z)$ is a polynomial of degree $t(p-1)$. The generating function $F(z)$ is rational by Theorem~\ref{thm:formF}.  \\
	2) If $(w_i,w_j)_z=0$ for all $i\ne j$, then $r_\F(z)=\sum_{i=1}^t 1/(w_i,w_i)_z$.
\end{rems}

The authors have proved in~\cite{Parry} that the Perron root $\theta$ of the subshift $\Sigma_\F$ is the largest positive real pole of the generating function $F(z)$. The statement of the result is given below and we refer to~\cite[Theorem 4.1]{Parry} for the proof. 

\begin{thm}\label{thm:same_roots}
	Consider an irreducible subshift of finite type $\Sigma_\F$ and let $\ln\theta$ be its topological entropy. The (rational) generating function $F(z)$ is analytic outside the closed disk centered at the origin with radius $\theta$. Moreover, $\theta$ is a simple pole of $F$. 
\end{thm}

\begin{rems}
	1) Theorem~\ref{thm:same_roots} can be extended to reducible subshifts provided there is at least one non-zero irreducible component. \\
	2)  From Theorems~\ref{thm:formF} and~\ref{thm:same_roots}, the Perron root $\theta$ is the largest real zero in modulus of the rational function $z-q+r(z)$.
\end{rems}

\begin{cor}\label{cor:e_rate}
	Consider an irreducible subshift of finite type $\Sigma_{\F}$. Let $\G$ be a finite collection of words with symbols from $\Sigma$. Then $\rho(\G;\Sigma_{\F})=-\ln(\lambda/\theta)$, where $\theta$ and $\lambda$ are the largest real poles of the generating functions corresponding to the collections $\F$ and $\F\cup\G$, respectively.
\end{cor}
\begin{proof}
	This result is immediate from Theorems~\ref{thm:esc_rate} and~\ref{thm:same_roots}.
\end{proof}

\begin{rem}\label{rem:Markov}
	By Theorem~\ref{thm:formF} and Corollary~\ref{cor:e_rate}, if the underlying subshift of finite type is induced with the Parry measure then the escape rates into two holes with identical correlation function $r(z)$ are the same. The following example illustrates that this may not be the case when the space is induced with a general Markov measure that is not the Parry measure. Consider the full shift on two symbols equipped with the Markov measure corresponding to the stochastic matrix $\begin{pmatrix}
		1/3&2/3\\
		1/2&1/2
	\end{pmatrix}$ (we refer to~\cite{Parry64} for the definition of a Markov measure). Consider two collections $\G_1=\{00,11\}$ and $\G_2=\{01,10\}$. Let $H_1$ and $H_2$ be holes corresponding to $\G_1$ and $\G_2$, respectively. The correlation functions $r_{\G_1}(z), r_{\G_2}(z)$ satisfy $r_{\G_1}(z)=r_{\G_2}(z)=(z+1)/2$, for all $z$. However the escape rates are $\rho(\G_1)=\ln(3)/2$ and $\rho(\G_2)=\ln(2)$.  
\end{rem}

\subsection{Minimal period of a hole}

\begin{Def}\label{def:min-period}
	Let $\Sigma_\F$ be a subshift of finite type and let $\G$ be a finite collection of allowed words in $\Sigma_\F$. Let $\tau_{\G}$ be the minimum of the period of periodic points in the hole $H_{\G}=\bigcup_{u\in\G} C_{u}$, under the shift map. We call $\tau_{\G}$, the \textit{minimal period of the hole $H_{\G}$}. If $\G$ contains only one word $u$, then we use the shorthand $\tau_u$ for $\tau_\G$. 
\end{Def}

\begin{rems}\label{rem:prt_mp}
	1) The minimal period of the hole $H_\G$ is given by $\tau_{\G}=\min_{u\in\G}\tau_{u}$. For any finite word $u$ of length $p$, $1\leq\tau_u\leq p$. If $1\leq \tau_{u}\leq p-1$, then the second largest degree of $(u,u)_z$ is $z^{p-1-\tau_{u}}$. Moreover, $(u,u)_z=z^{p-1}$, if $\tau_{u}=p$.\\
	2) When all the cross-correlation polynomials between the forbidden words in $\G$ are zero, $\tau_{\G}$ is the Poincar\'e recurrence time of $H_{\G}$. This need not be the case when the cross-correlation polynomials are non-zero. For example, when $\G=\{01000,10000\}$, the Poincar\'e recurrence time of $H_{\G}$ is 1, whereas $\tau_{\G}=4$.
\end{rems}

\section{Statement of main results}\label{sec:main-results}
Here we state the main results of this paper; proofs are presented in later sections. The first set of results are for a full shift $\Sigma^\mathbb{N}$ on $q$ symbols and discuss comparison of the escape rates into two holes, each corresponding to a union of two cylinders based at words of equal length $p\ge 2$. The techniques used in the proofs are different when $p=2$ and when $p\geq 3$, hence we have separated these results. Let $r_{\G_1}(z)$ and $r_{\G_2}(z)$ denote the rational functions corresponding to the collections $\G_1$ and $\G_2$, respectively, as described in Theorem~\ref{thm:formF}.

\begin{customthm}{1}
	Let $q\geq 2$ and $\G_1$, $\G_2$ be collections each having two words of length two ($p=2$). If $r_{\G_1}(2)<r_{\G_2}(2)$, then $\rho(\G_1)<\rho(\G_2)$.
\end{customthm}

\begin{customthm}{2}
	\label{thm:3}
	Let $p\geq 3$, $q>3p^2+2$, and $\G_1$, $\G_2$ be collections each having two words of length $p$. If $r_{\G_1}(3p^2+2)<r_{\G_2}(3p^2+2)$, then $\rho(\G_1)<\rho(\G_2)$.
\end{customthm}

The above results give the relationship between the escape rate and the rational function $r(z)$, and are proved in Section~\ref{subsec:esc_r}. 

The next result is also for a full shift $\Sigma^\mathbb{N}$ and discusses the relationship between the minimal period of the hole and the escape rate when the hole corresponds to a union of two cylinders based at words of length $p$ with zero cross-correlations. In this case, the minimal period is same as the Poincar\'e recurrence time, as was pointed out earlier. Let $\tau_{\G_1}$ and $\tau_{\G_2}$ denote the minimal period of the hole $H_{\G_1}$ and $H_{\G_2}$, respectively. The proof of the next result is given in Section~\ref{subsec:esc_min}.

\begin{customthm}{4}
	Let $\G_1=\{u_1,u_2\}$ and $\G_2=\{w_1,w_2\}$, where $u_1$, $u_2$, $w_1$, and $w_2$, are words of equal length $p$ with $(u_i,u_j)_z=(w_i,w_j)_z=0$, for $1\leq i\neq j\leq 2$.	For $p=2$ and $q\geq 2$, or for $p\geq 3$ and $q\geq 5$, if $\tau_{\G_1}<\tau_{\G_2}$, then $\rho(\G_1)<\rho(\G_2)$. 
\end{customthm} 

Finally we generalize the results stated above to the case when holes are unions of more than two cylinders. 

\begin{customthm}{5}
	Suppose $\G_1$ and $\G_2$ are finite collections of $t$ words each of length $p$ with symbols from $\Sigma$. Then for $p\geq 3$, there exists a positive constant $D=D(t,p)$ such that for all $q> D$, if $r_{\G_1}(D)<r_{\G_2}(D)$, then $\rho(\G_1)<\rho(\G_2)$.
\end{customthm}

\begin{customthm}{6}
	Suppose $t\geq 2$ and $q\geq 2$ satisfy $(q-1)(t+1)-t q\left(1+1/q\right)^{t-1}\geq 0$.
	Let $\G_1=\{u_1,u_2,\ldots,u_t\}$ and $\G_2=\{w_1,w_2,\ldots,w_t\}$ be such that $u_i,w_i$ are words of equal length $p$ with $(u_i,u_j)_z=(w_i,w_j)_z=0$ for $1\leq i\neq j\leq t$. If $\tau_{\G_1}<\tau_{\G_2}$, then $\rho(\G_1)<\rho(\G_2)$. 
	The above inequality holds true for $q\geq t 2^{t-1}+1$.
\end{customthm}

\noindent These results are proved in Sections~\ref{subsec:esc_r_1} and~\ref{subsec:gen_cor_period}, respectively.

\section{Full shift with union of two cylinders as hole}\label{sec:full_two}
In this section, we consider a full shift $\Sigma^\mathbb{N}$ on $q$ symbols with a hole which is a union of two cylinders based at words of identical length. In Theorems~\ref{thm:esc_r_1} and~\ref{thm:esc_r_2}, we explore the relationship between the escape rate and the corresponding rational function $r(z)$. Furthermore, in Theorem~\ref{thm:esc_min1}, we discuss the relationship between the escape rate and the minimal period of the hole. Our results easily extend to the case when the hole is a union of more than two cylinders, which is discussed in Section~\ref{sec:full_morethantwo}. All the results are applicable to maps conjugate to a full shift, for instance, to a product of expansive Markov maps where the individual maps are conjugate to a full shift as discussed in~\cite{Product}.  

\subsection{Relationship between the escape rate and the rational function $r(z)$}\label{subsec:esc_r}
For a finite collection $\G$ of words with symbols from $\Sigma$, let $H_{\G}=\bigcup_{w\in\G} C_w$, the union of cylinders based at words from $\G$. We say that $H_\G$ is a hole corresponding to the collection $\G$. Let $\rho(\G)$ denote the escape rate of the shift map on the full shift $\Sigma^\N$ into the hole $H_\G$. We recall the following result for a full shift with a single cylinder as a hole. 

\begin{thm}\label{thm:one-word}\cite[Lemma 4.5.1]{BY}
	Let $\G_1=\{w_1\}$ and $\G_1=\{w_2\}$, where $w_1$ and $w_2$ are words of identical length. If $(w_1,w_1)_q<(w_2,w_2)_q$, then $\rho(\G_2)<\rho(\G_1)$, where $(w,w)_q$ is the value of $(w,w)_z$ evaluated at $z=q$.
\end{thm}

\noindent Since the coefficients of the autocorrelation polynomials are either 0 or 1, if $(w_1,w_1)_2>(w_2,w_2)_2$ then $(w_1,w_1)_q>(w_2,w_2)_q$ for all $q\geq 2$, and hence by the preceding result, $\rho(\G_2)>\rho(\G_1)$. Consequently, for two fixed holes, the relation between the corresponding escape rates remain unchanged when the size of the symbol set is changed. Our aim is to extend this result to the situation where the hole is a union of cylinders based at words of identical length.

Let us consider the case when the hole corresponds to two words of identical length. Let $\G_1=\{u_1,w_1\}$ and $\G_2=\{u_2,w_2\}$, where $u_1\neq w_1,u_2\neq w_2$ are all words of length $p\ge 2$. We would like to compare the escape rates $\rho(\G_1)$ and $\rho(\G_2)$ of the shift map on the full shift $\Sigma^\N$ into the holes $H_{\G_1}$ and $H_{\G_2}$ corresponding to the collections $\G_1$ and $\G_2$, respectively. 

For $i=1,2$, let $f_{i}(k)$ denote the number of words of length $k$ with symbols from $\Sigma$ which do not contain any of the words from the collection $\G_i$ as subwords, and let the generating function be $F_i(z)=\sum_{n=0}^\infty f_i(n)z^{-n}$. By Theorem~\ref{thm:formF}, \[
F_{i}(z)=\dfrac{z}{z-q+r_{\G_i}(z)}, \]
where 
\[
r_{\G_i}(z)=\dfrac{(u_i,u_i)_z+(w_i,w_i)_z-(u_i,w_i)_z-(w_i,u_i)_z}{(u_i,u_i)_z(w_i,w_i)_z-(u_i,w_i)_z(w_i,u_i)_z},
\]
is the sum of the entries of the inverse of the correlation matrix function corresponding to $\G_i$.

\begin{table}[h] 
	\centering
	\caption{Escape rate values for $p=2$ and $q=2,3,\ldots,10$. For all real $z>1$, $r_{\G_1}(z)<r_{\G_2}(z)<r_{\G_3}(z)<r_{\G_4}(z)$.}
	\label{table:1}
	\begin{footnotesize}
		\begin{tabular}{*5c} 
			\hline
			$i$ & 1&2&3&4\\ \hline
			$\G_i$ & $\{aa,bb\}$	& $\{ab,ca\}$ & $\{aa,bc\}$ & $\{ab,cd\} / \{ab,ac\}$\\ \hline
			\backslashbox{$q$\kern-2em}{\kern-1em $r_{\G_i}(z)$} & $\dfrac{2}{z+1}$ & $\dfrac{2z-1}{z^2}$ & $\dfrac{2z+1}{z(z+1)}$ & $\dfrac{2}{z}$\\ \hline
			2&0.6931&\dots&\dots&\dots\\
			3&0.2172&0.2550& 0.2890&0.4055 \\
			4&0.1161& 0.1302&0.1361&0.1583\\
			5&0.0735&0.0804&0.0824&0.0918\\
			6&0.0510&0.0550&0.0559&0.0609\\
			7&0.0376&0.0401&0.0406&0.0436\\
			8&0.0289&0.0306&0.0309&0.0328\\ 
			9&0.023&0.0241&0.0243&0.0257\\
			10&0.0187&0.0195&0.0196&0.0206\\
			\hline
		\end{tabular}
	\end{footnotesize}
\end{table}

The escape rates into the holes corresponding to two forbidden words are given in Tables~\ref{table:1},~\ref{table:2} and~\ref{table:1general}. These tables show the relationship between the rational functions $r(z)$ and the escape rate for $p=2,3,5$, and for certain values of $q$. In the tables, $a, b, c, d, e, f$ are distinct symbols from $\Sigma=\{0,1,\ldots,q-1\}$. 

Let us make a few observations at this point. In Table~\ref{table:2} where $p=3$, consider the collections $\G_2=\{aaa,aba\}$ and $\G_3=\{abb,bba\}$. For all real $z\ge 2$, $r_{\G_2}(z)<r_{\G_3}(z)$. When $q=2$, $\rho(\G_2)=\rho(\G_3)$, however, for $q=3,4,5$, $\rho(\G_2)<\rho(\G_3)$. Also, in Table~\ref{table:1general} where $p=5$, consider the collections $\G_5=\{bbabb,bbbab,bbbba\}$ and $\G_6=\{abbbb,bbbba,bbabb\}$. For all real $z\ge 3$, $r_{\G_5}(z)<r_{\G_6}(z)$. When $q=2$, $\rho(\G_5)>\rho(\G_6)$, however, for $q=3$, $\rho(\G_5)<\rho(\G_6)$. 

These tables suggest the following: a) for any pair $i,j$, there exists a number $q_0$ such that $r_{\G_i}(z)-r_{\G_j}(z)$ has the same sign for all real $z\ge q_0$, and b) $\rho(\G_i)<\rho(\G_j)$ if $r_{\G_i}(q_0)<r_{\G_j}(q_0)$. Indeed we have Theorems~\ref{thm:esc_r_1} and~\ref{thm:esc_r_2} dealing with the cases $p=2$ and $p\ge 3$, respectively.

\begin{thm} \label{thm:esc_r_1}
	Let $q\ge 2$ and $\G_1$, $\G_2$ be the collections each having two words of length 2. If $r_{\G_1}(2)<r_{\G_2}(2)$, then $\rho(\G_1)<\rho(\G_2)$.
\end{thm}

\begin{proof} 
	The only possible collections with two words of length 2 are $\{aa,bb\}$, $\{aa,ab\}$, $\{ab,ba\}$, $\{ab,ca\}$, $\{ab,bc\}$, $\{aa,bc\}$, $\{ab,cd\}$, and $\{ab,ac\}$. Consequently there are four possible forms for $r(z)$, given by
	\[
	r_1(z)=\dfrac{2}{z+1},\ r_2(z)=\dfrac{2z-1}{z^2},\ r_3(z)=\dfrac{2z+1}{z(z+1)},\ r_4(z)=\dfrac{2}{z}.
	\] 
	The collections $\{ab,ca\}$, $\{ab,bc\}$, $\{aa,bc\}$, $\{ab,cd\}$, and $\{ab,ac\}$, and hence the rational functions $r_2$, $r_3$, and $r_4$ are not possible when there are only two symbols ($q=2$), thus there is nothing to prove in this case. Therefore we assume that $q\ge 3$. Here $r_1(z)<r_2(z)<r_3(z)<r_4(z)$, for all real $z>1$. We need to prove that $\rho(\G_1)<\rho(\G_2)<\rho(\G_3)<\rho(\G_4)$. By Theorems~\ref{thm:esc_rate} and~\ref{thm:same_roots}, this is equivalent to proving that $\lambda_1>\lambda_2>\lambda_3>\lambda_4$, where $\lambda_i$ is the largest real zero of $z-q+r_i(z)$.
	
	Note that for $i=1,\ldots,4$, the largest real pole of $F_i(z)$ is same as the largest real root of $p_i(z)$, where $p_1(z)=z^2-(q-1)z-(q-2)$, $p_2(z)=z^3-qz^2+2z-1$, $p_3(z)=z^3-(q-1)z^2-(q-2)z+1$, and $p_4(z)=z^2-qz+2$. For each $i$, $p_i(q-1)<0$ and $p_i(q)>0$. Hence, $p_i$ has a real root $\lambda_i\in (q-1,q)$. We prove that $\lambda_i$ is a simple root of $p_i$ with largest modulus. 
	
	Let us analyze each polynomial separately. Since $p_1(-1)>0$, there exists a real root inside $(-1,q-1)$. Thus $\lambda_1$ is simple and has largest modulus since $p_1$ is a degree 2 polynomial and $q\ge 3$. For $q\ge 4$, at $|z|=q-1$, note that $q|z|^2>|z|^3+2|z|+1$, and hence $p_2(z)$ has two roots inside the ball of radius $q-1$, by Rouch\'e's theorem. Thus $\lambda_2$ is the simple largest root of $p_2$. At $q=3$, it is immediate that $\lambda_2$ is the largest root of $p_2(z)=z^3-3z^2+2z-1$. Since $p_3(-1)<0$, $p_3(0)>0$, and $p_3$ is a degree 3 polynomial, $\lambda_3$ is the largest root of $p_3$. Finally, since $p_4(0)>0$ and $p_4$ is a degree 2 polynomial, $\lambda_4$ is the largest root of $p_4$.
	
	Note that $\lambda_2p_1(\lambda_2)=\lambda_2^3-(q-1)\lambda_2^2-(q-2)\lambda_2=\lambda_2^2-q\lambda_2+1 <0$, since $\lambda_2(\lambda_2^2-q\lambda_2+1)=1-\lambda_2<0$. Hence $p_1(\lambda_2)<0$ which implies $\lambda_2<\lambda_1$.
	Similarly $p_3(\lambda_2)>0$ implies $\lambda_3<\lambda_2$, and $p_3(\lambda_4)<0$ implies $\lambda_4<\lambda_3$, for $q\ge 3$. Thus we have the required result using Corollary~\ref{cor:e_rate}.
\end{proof}

\begin{table}
	\centering
	\caption{Escape rate values for $p=3$ and $q=2,3,4,5$. Here $r_{\G_i}(z)<r_{\G_{i+1}}(z)$, for all $z>4$.}
	\label{table:2}
	\begin{footnotesize}
		\begin{tabular}{c|c|c|cccc} 
			\hline
			&&&\multicolumn{4}{c}{$\rho(\G_i)$}\\
			& &$r_{\G_i}(z)$&$q=2$&$q=3$&$q=4$&$q=5$\\
			\hline
			$\G_1$&$\{aaa,bbb\}$/&\\& $\{aaa,aab\}$/&$\dfrac{2}{z^2+z+1}$&0.212&0.0579&0.0252&0.0133\\& $\{aaa,baa\}$& \\\hline
			
			$\G_2$&$\{aaa,aba\}$&$\dfrac{2z+1}{z^3+z^2+2z+1}$&0.212&0.0616 &0.0268&0.0141\\\hline
			
			$\G_3$& $\{abb,bba\}$&$\dfrac{2z^2-z-2}{z^4-z-1}$&0.212& 0.0645& 0.0275&0.01436\\\hline
			
			$\G_4$& $\{abb,bab\}$&$\dfrac{2z-1}{z^3+z-1}$&0.311&0.0664& 0.0278&0.0144\\\hline
			
			$\G_5$&$\{aaa,bab\}$/&$\dfrac{2z^2+z+2}{(z^2+z+1)(z^2+1)}$&0.311&0.0686&0.0284&0.01464 \\& $\{aaa,bcb\}$& \\\hline
			
			$\G_6$&$\{abc,bcb\}$&$\dfrac{2z^2-z+1}{z^4+z^2}$&\dots &0.0691 &0.0285&0.01465\\\hline
			
			$\G_7$&$\{aaa,abb\}/$&\\& $\{aaa,bba\}/$&$\dfrac{2z+1}{z(z^2+z+1)}$&0.412 &0.0700 &0.0286&0.01469\\&$\{aaa,abc\}$/&\\ &$\{aaa,bca\}$&\\\hline
			
			$\G_8$&$\{abb,aab\}$/& \\&$\{abc,bcd\}$/&$\dfrac{2z-1}{z^3}$&0.693&0.0729 &0.0293&0.0149\\ &$\{abc,bcc\}$&\\\hline
			
			$\G_9$&$\{abc,ddd\}$/&$\dfrac{2z^2+z+1}{z^2(z^2+z+1)}$&\dots& 0.0746 &0.0295&0.0150 \\ &$\{abc,bbb\}$&\\\hline
			
			$\G_{10}$&$\{aba,aca\}$&$\dfrac{2}{z^2+2}$&\dots &0.0703& 0.0297&0.0154\\\hline
			
			$\G_{11}$&$\{abb,aba\}$/&\\ &$\{abc,cdc\}$/&$\dfrac{2}{z^2+1}$&0.693&0.0800&0.0317&0.160\\ &$\{abc,cba\}$/&\\ &$\{abc,cda\}$ &\\\hline
			
			$\G_{12}$&$\{abc,cde\}$ /&$\dfrac{2z^2-1}{z^4}$&\dots&0.0857 &0.0328 &0.01632\\ &$\{abc,cbb\}$ &\\\hline
			
			$\G_{13}$&$\{abc,ded\}$/&$\dfrac{2z^2+1}{z^4+z^2}$&\dots&\dots &0.0329&0.01634 \\ &$\{abc,dad\}$ & \\\hline
			
			$\G_{14}$&$\{abc,def\}$/&\\ &$\{abc,dcb\}$/&$\dfrac{2}{z^2}$&\dots&\dots &0.0340&0.0167\\ &$\{abc,dee\}$ &\\ \hline	\end{tabular}
	\end{footnotesize}
\end{table}

\begin{table}[h]
	\centering
	\caption{Escape rates corresponding to few collections of three words each of length five ($p=5$). Here $r_{\G_i}(z)<r_{\G_{i+1}}(z)$, for all real $z>3$.} 
	\label{table:1general}
	\begin{footnotesize}
		\begin{tabular}{c|c|c|cc} 
			\hline
			$i$ & $\G_i$&$r_{\G_i}(z)$&\multicolumn{2}{c}{$\rho(\G_i)$}\\
			&& &$q=2$&$q=3$\\\hline
			1&$\{abbbb,bbbba,bbbbb\}$&$\dfrac{-2 - z - z^2 - z^3 + 3 z^4}{-\sum_{j=0}^3 z^j + \sum_{j=4}^8 z^j}$&0.037& 0.0071\\\hline
			2&$\{aaaaa,ababa,babab\}$&$ \dfrac{2 + z + 3 z^5}{z^5+\sum_{j=0}^9 z^j  }$&0.058& 0.0085\\\hline
			3&$\{aaaaa,bbbbb,ccccc\}$&$\dfrac{3}{ \sum_{j=0}^4 z^j}$&\dots & 0.0086\\\hline
			4&$\{aaaaa,abaaa,bbbbb\}$&$\dfrac{2 + 5 z + 4 z^2 + 4 z^3 + 4 z^4 + 3 z^5}{(\sum_{j=0}^4 z^j)+(1 + 2 z + z^2 + z^3 + z^4 + z^5)}$&0.070& 0.0095\\\hline
			5&$\{bbabb,bbbab,bbbba\}$&$\dfrac{-2 + 3 z}{-2 + z + z^2 + z^5}$&0.072& 0.0096 \\\hline		
			6&$\{abbbb,bbbba,bbabb\}$&$\dfrac{-2 - 3 z + z^3 - 3 z^4 - z^5 + 3 z^6}{-1 - 2 z - z^2 - 2 z^4 - 2 z^5 + z^6 + z^7 + z^{10}}$&0.060& 0.0097\\
			\hline
		\end{tabular}
	\end{footnotesize}
	
\end{table}

Next we prove a series of lemmas required to obtain the main result Theorem~\ref{thm:esc_r_2}, which relates the escape rate and the rational function $r(z)$ when a hole corresponds to a collection of two words of identical length. Let us first some few notations. Let $u,w$ be two distinct words both of length $p\ge 2$ with symbols from $\Sigma=\{0,1,\ldots,q-1\}$, and let $\G=\{u,w\}$.
Define
\begin{eqnarray*}
	a(z):=(u,u)_z, \ b(z):=(w,u)_z, \ c(z):=(u,w)_z,\ d(z):=(w,w)_z,\\
	\D(z):=a(z)d(z)-b(z)c(z), \ \mathcal{S}(z):=a(z)+d(z)-b(z)-c(z).
\end{eqnarray*}
Then the generating function $F(z)$ corresponding to the collection $\G$, as described in Theorem~\ref{thm:formF}, is given by 
\begin{eqnarray}\label{eq:F}
	F(z)=\dfrac{z}{z-q+r(z)}=\dfrac{z\D(z)}{(z-q)\D(z)+\mathcal{S}(z)},
\end{eqnarray}
where $r(z)=\dfrac{\mathcal{S}(z)}{\D(z)}$ is the rational function given by the sum of entries of $\M(z)^{-1}$, where $\M(z)=\begin{pmatrix}
	a(z)&b(z)\\c(z)&d(z)
\end{pmatrix}$.

Observe that $a,b,c,d$ are polynomial functions with coefficients either 0 or 1. The degree of both $a(z)$ and $d(z)$ is $p-1$, whereas the degree of both $b(z)$ and $c(z)$ is at most $p-2$. Hence $\D(z)$ and $\mathcal{S}(z)$ have degree $2(p-1)$ and $p-1$, respectively. 

We show that for real $z$, $\D(z)>0$, and thus from Theorem~\ref{thm:same_roots}, the Perron root of the adjacency matrix for $\Sigma_\G$ is same as the largest real root of $(z-q)\D(z)+\mathcal{S}(z)$. The following two lemmas give the location of this root on the real line. 

\begin{lemma} \label{Lemma:1}
	For $p\ge 3$ and $q\ge 5$, the polynomial $(z-q)\D(z)+\mathcal{S}(z)$ has exactly one root outside the disk $\vert z\vert<4$. 
\end{lemma}
\begin{proof}
	Since $a,b,c,d$ are polynomial functions with coefficients either 0 or 1, and the degree of both $a(z)$ and $d(z)$ is $p-1$, whereas the degree of both $b(z)$ and $c(z)$ is at most $p-2$,
	\[
	\vert a(z)\vert, \vert d(z)\vert \ge \vert z\vert^{p-1}-\sum_{i=0}^{p-2}\vert z\vert^i \ge \vert z\vert^{p-1} \dfrac{\vert z\vert-2}{\vert z\vert-1},\]
	\[ \vert c(z)\vert, \vert b(z)\vert \le  \sum_{i=0}^{p-2}\vert z\vert^i\le \dfrac{\vert z\vert^{p-1}}{\vert z\vert-1}.
	\]
	Hence, on the circle $\vert z\vert=4$,
	\begin{eqnarray*}
		\vert \D(z)\vert &\ge & \vert a(z)\vert\vert d(z)\vert-\vert c(z)\vert\vert b(z)\vert \ge \vert z\vert^{2p-2}\left(\dfrac{\vert z\vert-2}{\vert z\vert-1}\right)^2-\dfrac{\vert z\vert^{2p-2}}{(\vert z\vert -1)^2}\\
		&=&\vert z\vert^{2p-2} \left(1-\dfrac{2}{\vert z\vert-1}\right) = \dfrac{1}{48} 4^{2p}>0,\\
		\vert \mathcal{S}(z)\vert&\le& \vert a(z)\vert+\vert d(z)\vert+\vert c(z)\vert+\vert b(z)\vert\le  \dfrac{2}{\vert z\vert -1}\vert z\vert^{p-1}\left(\vert z \vert+1\right)= \dfrac{5}{6} 4^p.
	\end{eqnarray*}
	Since $q\ge 5$ and $p\ge 3$, on the circle $\vert z\vert=4$,
	\begin{eqnarray*}
		\dfrac{\vert z-q\vert\vert \D(z)\vert}{\vert \mathcal{S}(z)\vert}\ge \dfrac{\vert \D(z)\vert}{\vert \mathcal{S}(z)\vert}\ge \dfrac{1}{40}4^p>1.
	\end{eqnarray*}
	Hence $\vert z-q\vert\vert \D(z)\vert>\vert \mathcal{S}(z)\vert$, on $\vert z\vert=4$. By Rouch\'e's theorem, both $(z-q)\D(z)$ and $(z-q)\D(z)+\mathcal{S}(z)$, being of identical degrees, have the same number of zeros in the region $\vert z\vert\ge 4$.
	
	Since $\vert \D(z)\vert>0$, $z=q$ is the only zero of $(z-q)\D(z)$, for $\vert z\vert\ge 4$, which gives the desired result.
\end{proof}

Let $\lambda\ge 4$ denote the unique real root of $(z-q)\D(z)+\mathcal{S}(z)$ obtained in Lemma~\ref{Lemma:1}. By Theorem~\ref{thm:same_roots}, $\lambda$ is same as the Perron root of the adjacency matrix of $\Sigma_\G$, since $\D(\lambda)> 0$.
By the Perron-Frobenius theorem, $\lambda\le q$. The following lemma discusses how close the Perron root $\lambda$ is to $q$. 

\begin{lemma}\label{Lemma:2}
	For $p\ge 3$ and $q\ge 5$, $q-\lambda=O(q^{-p+2})$.
\end{lemma}
\begin{proof}	
	Since $\mathcal{S}(q)=a(q)+d(q)-b(q)-c(q)>2q^{p-1}-2\dfrac{q^{p-1}-1}{q-1}>0$, at $z=q$, $(q-z)\D(z)-\mathcal{S}(z)=-\mathcal{S}(q)<0$. For real $x\ge 2$, 
	\begin{eqnarray*}
		(q-x)\D(x)-\mathcal{S}(x)&\ge& (q-x)\left[x^{2p-2}-\dfrac{x^{2p-2}}{(x-1)^2}\right]-2\dfrac{x^p}{x-1}\\
		&>& \dfrac{x^{p+1}}{(x-1)^2} ((q-x)(x-2)x^{p-2}-2).
	\end{eqnarray*}
	Moreover at $x=q-mq^{-p+2}$,
	\begin{eqnarray*}
		(q-x)(x-2)x^{p-2}-2 &=& mq^{-p+2} q^{p-2}(1-mq^{-p+1})^{p-2}(q-mq^{-p+2}-2)-2\\
		&\ge& m(1-m(p-2)q^{-p+1})(q-mq^{-p+2}-2)-2,
	\end{eqnarray*}
	for $p\ge 3$, $q\ge 5$, since $(1-x)^n\ge 1-nx$, for $x>0$ and $n\ge 1$. \\
	Further since for $p\ge 3$, $\dfrac{p-2}{q^{p-1}}<\dfrac{p-1}{5^{p-1}}<\dfrac{1}{8}$,
	\[
	m\left(1-m(p-2)q^{-p+1}\right)\left(q-mq^{-p+2}-2\right)-2>m\left(1-\dfrac{m}{8}\right)\left(3-\dfrac{m}{25}\right)-2.
	\]
	For all values of $m$ lying in the interval $[0.8,1]$, the last term is positive. Hence we have $(q-x)\D(x)-\mathcal{S}(x)>0$, for $x=q-mq^{-p+2}$. Thus
	\begin{eqnarray*}
		q-q^{-p+2}<	q-mq^{-p+2} <\lambda < q, 
	\end{eqnarray*}
	and hence the result follows.
\end{proof}

\begin{rem}
	An ancillary result of the preceding lemma is that $q-1<\lambda<q$.
\end{rem}

The above result leads to the following corollary which gives the existence of an arbitrarily large hole into which the escape rate is arbitrarily small.

\begin{cor}
	Let $q\ge 5$ and $x,y\in\Sigma^{\N}$. For any given $\epsilon>0$ and $\delta>0$, there exists a collection $\G$ of words such that $x,y\in H_{\G}$, $\mu(H_{\G})>1-\epsilon$, and $\rho(\G)<\delta$.
\end{cor}
\begin{proof}
	Let $x=x_1x_2\ldots$, $y=y_1y_2\ldots$, $\G_p=\{u_p=x_1x_2\ldots x_p,w_p=y_1y_2\ldots y_p\}$, and $\D_p(z)$ and $\mathcal{S}_p(z)$ be as in~\eqref{eq:F} such that the rational function $r_{p}(z)=\dfrac{\mathcal{S}_p(z)}{\D_p(z)}$.
	
	By Lemma~\ref{Lemma:2}, for $p\ge 3$, the largest root $\lambda_p$ of the polynomial $(z-q)\D_p(z)+\mathcal{S}_p(z)$ satisfies
	\[
	q-q^{-p+2}\le \lambda_p\le q.
	\]
	Hence $\rho(\G_p)=-\ln\left(\lambda_p/q\right)\rightarrow 0$, as $p\rightarrow\infty$. Choose $p$ large enough so that $\rho(\G_p)<\delta$. Since $\sigma$ is ergodic with respect to the Parry measure $\mu$, there exists $n$ such that $\mu\left(\bigcup_{i=0}^n\sigma^{-i}H_{\G_p}\right)>1-\epsilon$. Let $H_\G=\bigcup_{i=0}^n\sigma^{-i}H_{\G_p}$. Using~\cite[Proposition 2.3.2]{BY}, $\rho(\G)=\rho(\G_p)$. This completes the proof. 
\end{proof}

\noindent The next result will be used in the proof of Theorem~\ref{thm:esc_r_2}.

\begin{prop}\label{thm:esc_r_x}
	Let $p\ge 3, q\ge 5$, and $\G_1$, $\G_2$ be collections each having two words of length $p$. If $r_{\G_1}(x)<r_{\G_2}(x)$ for all $x\in (q-q^{-p+2},q)$, then $\rho(\G_1)<\rho(\G_2)$.
\end{prop}
\begin{proof}
	Let $\D_i(x)$, $\mathcal{S}_i(x)$, $r_{\G_i}(x)=\dfrac{\mathcal{S}_i(x)}{\D_i(x)}$ be as defined earlier, for the collection $\G_i$, $i=1,2$. Let $\lambda_i\in(q-q^{-p+2},q)$ be the largest root of $(q-x)\D_i(x)-\mathcal{S}_i(x)$, for $i=1,2$. Then $q-\lambda_i-r_{\G_i}(\lambda_i)=0$, $i=1,2$.
	
	Since $r_{\G_1}(x)<r_{\G_2}(x)$, for all $x\in(q-q^{-p+2},q)$, we have, in particular, $r_{\G_1}(\lambda_1)<r_{\G_2}(\lambda_1)$ and $r_{\G_1}(\lambda_2)<r_{\G_2}(\lambda_2)$. Consider the equation $q-x-r_{\G_2}(x)$. At $x=\lambda_1$, 
	\begin{eqnarray*}
		q-\lambda_1-r_{\G_2}(\lambda_1)&=&q-\lambda_1-r_{\G_2}(\lambda_1)-\left(q-\lambda_1-r_{\G_1}(\lambda_1)\right)\\
		&=& r_{\G_1}(\lambda_1)-r_{\G_2}(\lambda_1)<0.
	\end{eqnarray*} 
	
	Since at $x=q-q^{-p+2}$, $q-x-r_{\G_2}(x)>0$, and $\lambda_2$ is the unique root of $q-x-r_{\G_2}(x)$ greater than $q-q^{-p+2}$, we obtain $\lambda_2<\lambda_1$. Hence $\rho(\G_1)<\rho(\G_2)$ using Corollary~\ref{cor:e_rate}.
\end{proof}

We now prove our main theorem which relates the escape rate and the rational function when the hole is a union of two cylinders based at words of identical length.

\begin{thm}\label{thm:esc_r_2}
	Let $p\ge 3$, $q>3p^2+2$, and $\G_1$, $\G_2$ be the collections each having two words of length $p$. If $r_{\G_1}(3p^2+2)<r_{\G_2}(3p^2+2)$, then $\rho(\G_1)<\rho(\G_2)$.
\end{thm}
\begin{proof}
	For $i=1,2$, let $\G_i=\{u_i,w_i\}$. Define $f_i(z):=(u_i,u_i)_z$, $g_i(z):=(w_i,w_i)_z$, $h_i(z):=(u_i,w_i)_z$, $k_i(z):=(w_i,u_i)_z$, $\D_i(z):=f_i(z)g_i(z)-h_i(z)k_i(z)$, $\mathcal{S}_i(z):=f_i(z)+g_i(z)-h_i(z)-k_i(z)$, and then the rational function $r_{\G_i}(z)=\dfrac{\mathcal{S}_i(z)}{\D_i(z)}$. 
	
	Let $\lambda_i$ be the largest real root of $(q-x)\D_i(x)-\mathcal{S}_i(x)$. We show that if $r_{\G_1}(3p^2+2)<r_{\G_2}(3p^2+2)$, then $r_{\G_1}(x)<r_{\G_2}(x)$ for all $x\ge 3p^2+2$. Proposition~\ref{thm:esc_r_x} will then imply that $\lambda_1>\lambda_2$, for all $q>3p^2+2$. 
	
	The zeros of $(s_1-s_2)(x)$ is same as the zeros of $(\D_2\mathcal{S}_1-\D_1\mathcal{S}_2)(x)$, which is a polynomial with integer coefficients. We show that the absolute value of each coefficient of $(\D_2\mathcal{S}_1-\D_1\mathcal{S}_2)(x)$ is bounded above by $3p^2+1$. Hence, by the Lagrange bound on the zeros of a polynomial, all the roots of $(s_1-s_2)(x)$ lie in the disc of radius $3p^2+2$.  
	
	Since $f_i,g_i,h_i,k_i$ are polynomials with coefficients either 0 or 1, the degree of $f_i,g_i$ is $p-1$, the degree of $h_i,k_i$ is at most $p-2$, if $\D_i(x)=\sum_{\ell=0}^{2p-2}a_{i,\ell} x^\ell$ and $\mathcal{S}_i(x)=\sum_{m=0}^{p-1}b_{i,m} x^m$, then
	\[ |a_{i,\ell}|\leq 
	\begin{cases} 
		\ell+1 & 0\leq \ell\leq p-1  \\
		2p-(\ell+1)    & p\leq \ell\leq 2p-2, 
	\end{cases}
	\]
	and $|b_{i,m}|\leq 2$ for $0\leq m\leq p-1$. Now, 
	\[
	\D_i\mathcal{S}_j(x)=\sum_{n=0}^{3p-3}\left(\sum_{\ell+m=n}a_{i,\ell}b_{j,m}\right)x^n.
	\]
	The maximum of the upper bounds for the coefficients $\sum_{\ell+m=n}a_{i,\ell}b_{j,m}$ is at $n=p+\dfrac{p}{2}-1$ when $p$ is even, and at $n=p+\dfrac{p-1}{2}-1$ when $p$ is odd, and the maximum values are $\dfrac{3p^2}{2}$ and $\dfrac{3p^2+1}{2}$, respectively. Hence, in general, for every $0\leq n\leq 3p-3$,
	\[
	\left|\sum_{l+m=n}a_{i,\ell}b_{j,m} \right | \leq \dfrac{3p^2+1}{2}.
	\]
	This implies that the absolute value of each coefficient of $\D_1\mathcal{S}_2-\D_2\mathcal{S}_1$ is bounded above by $3p^2+1$.\\
\end{proof}

\begin{rem}
	When the hole is a union of two cylinders based at words of identical length, then the cross-correlations between the words play an important role.
	
	For instance, consider $\G_1=\{012\}$, $\mathcal{H}_1=\{123\}$, $\G_2=\{102\}$, and $\mathcal{H}_2=\{333\}$. Here $\rho(\G_1)=\rho(\G_2)$ and $\rho(\mathcal{H}_1)>\rho(\mathcal{H}_2)$, but $\rho(\G_1\cup\mathcal{H}_1)<\rho(\G_2\cup\mathcal{H}_2)$.
\end{rem}
 So far in this section we have assumed that the hole is a union of two cylinders based at words of identical length. In the following result, we allow for a hole to be union of two cylinders based at words of different lengths.

\begin{thm}\label{thm:diff-length}
	Let $p_1,p_2\ge 3$, $q>3p_1p_2+2$, and $\G_1,\G_2$ be two collections of words each containing two words of lengths $p_1$ and $p_2$. If $r_{\G_1}(3p_1p_2+2)<r_{\G_2}(3p_1p_2+2)$, then $\rho(\G_1)<\rho(\G_2)$.
\end{thm}
\begin{proof}
	For a hole $H_\G$ corresponding to the collection $\G$ consisting of two words of lengths $p_1$ and $p_2$, let $a,b,c,d,\D,\mathcal{S}$ be as defined earlier. We follow the steps as in the proof of Theorem~\ref{thm:esc_r_2}. Without loss of generality, we assume that $p_1>p_2$. The following statements hold true for $p_1,p_2\ge 3$, $q\ge 5$.
	\begin{itemize}
		\item On $\vert z\vert=4$, $\dfrac{\vert z-q\vert\vert \D(z)\vert}{\vert \mathcal{S}(z)\vert}\ge \dfrac{1}{20}\dfrac{4^{p_1+p_2}}{4^{p_1}+4^{p_2}}>1$. Hence as in Lemma~\ref{Lemma:1}, the polynomial $(z-q)\D(z)+\mathcal{S}(z)$ has exactly one real positive root $\lambda>4$. 
		\item $q-\lambda=O(q^{-p_2+2})$.
		\item If $r_{\G_1}(x)<r_{\G_2}(x)$ for all $x\in (q-q^{-p_2+2},q)$, then $\rho(\G_1)<\rho(\G_2)$.
		\item The absolute value of all the zeros of $r_{\G_1}-r_{\G_2}$ has an upper bound equal to $3p_1p_2+2$.
	\end{itemize}
\end{proof}

\subsection{Relationship between the escape rate and minimal period of the hole}\label{subsec:esc_min}
We recall a result from~\cite{BY} which gives the relationship between the minimal period (Definition~\ref{def:min-period}) and the escape rate into a single cylinder on a full shift. 

\begin{thm}\cite[Theorem 4.0.8]{BY}
	Let $\G_1=\{u_1\}$ and $\G_2=\{u_2\}$, where $u_1$ and $u_2$ are words of identical length. If $\tau_{\G_1}<\tau_{\G_2}$, then $\rho(\G_1)<\rho(\G_2)$.
\end{thm}

This result is immediate from Theorem~\ref{thm:one-word} since $\tau_{\G_1}<\tau_{\G_2}$ implies $(u_1,u_1)_q > (u_2,u_2)_q$, which in turn implies $\rho(\G_1)<\rho(\G_2)$. Theorem~\ref{thm:esc_min1} extends this result to the case when holes are unions of two cylinders with an extra assumption that the cross-correlation polynomials between the forbidden words are zero. With this extra assumption, the minimal period is same as the Poincar\'e recurrence time. Later we show that this assumption is necessary. 

\begin{thm}\label{thm:esc_min1}
	Let $\G_1=\{u_1,u_2\}$ and $\G_2=\{w_1,w_2\}$, where $u_1$, $u_2$, $w_1$, and $w_2$ are words of equal length $p$. Let $(u_i,u_j)_z=(w_i,w_j)_z=0$, for $1\le i\neq j\le 2$. For $p=2$ and $q\ge 2$, or for $p\ge 3$ and $q\ge 5$, if $\tau_{\G_1}<\tau_{\G_2}$, then $\rho(\G_1)<\rho(\G_2)$. 
\end{thm}
\begin{proof} 
	It is easy to verify that the result holds true when $p=2$ and $q\ge 2$ since there are only four possible forms of rational functions as described in the proof of Theorem~\ref{thm:esc_r_1}, out of which only three have zero cross-correlation polynomials.
	
	Assume $p\ge 3,q\ge 5$. Using Proposition~\ref{thm:esc_r_x}, it is enough to show that $\tau_{\G_1}<\tau_{\G_2}$ implies $r_{\G_1}(x)<r_{\G_2}(x)$ for $x\in(q-q^{-p+2},q)$. We in fact prove that $r_{\G_1}(x)<r_{\G_2}(x)$ for all $x\ge 4$ (note that $q-q^{-p+2}\ge 4$). 
	
	Let $(u_i,u_i)_x=x^{p-1}+x^{p-1-\tau_{u_i}}+R_{u_i}(x)$, and $(w_i,w_i)_x=x^{p-1}+x^{p-1-\tau_{w_i}}+R_{w_i}(x)$, where $R_{u_i},R_{w_i}$ are the remainder terms having degrees less than $p-1-\tau_{u_i},p-1-\tau_{w_i}$, respectively, for $i=1,2$. Since all the cross-correlation polynomials between the words are zero, we have
	\[
	r_{\G_1}(x)=\dfrac{1}{(u_1,u_1)_x}+\dfrac{1}{(u_2,u_2)_x},\ \text{and } r_{\G_2}(x)=\dfrac{1}{(w_1,w_1)_x}+\dfrac{1}{(w_2,w_2)_x}.
	\]
	Without loss of generality, assume that $\tau_{u_1}\le\tau_{u_2}$, and $\tau_{w_1}\le\tau_{w_2}$ and hence $\tau_{\G_1}=\tau_{u_1}$ and $\tau_{\G_2}=\tau_{w_1}$. Since $\tau_{w_1}\le\tau_{w_2}$, we have $\dfrac{1}{(w_1,w_1)_x}\le\dfrac{1}{(w_2,w_2)_x}$. Hence for $x\ge 4$,
	
	\begin{eqnarray*}
		r_{\G_2}(x)- r_{\G_1}(x)&=&  \dfrac{1}{(w_1,w_1)_x}+\dfrac{1}{(w_2,w_2)_x}-\dfrac{1}{(u_1,u_1)_x}-\dfrac{1}{(u_2,u_2)_x}\\
		&\ge& \dfrac{2}{(w_1,w_1)_x}-\dfrac{1}{(u_1,u_1)_x}-\dfrac{1}{(u_2,u_2)_x}.
	\end{eqnarray*}
	Since $\tau_{u_1}<\tau_{w_1}$, we have $1\le\tau_{u_1}\le p-1$. 
	Let us consider the following cases:\\
	
	\noindent \textit{Case 1:} $1\le\tau_{w_1},\tau_{u_2}\le p-1$. In this case we have $(u_1,u_1)_x\ge x^{p-1}+x^{p-1-\tau_{u_1}}$, $(u_2,u_2)_x \ge x^{p-1}+x^{p-1-\tau_{u_2}}$, and $(w_1,w_1)_x\le x^{p-1}+x^{p-1-\tau_{w_1}}+x^{p-1-(\tau_{w_1}+1)}+\ldots+1$. Hence
	\begin{eqnarray}\label{ineq:1}
		r_{\G_2}(x)-	r_{\G_1}(x)
		&\ge& 
		\dfrac{2}{x^{p-1}+x^{p-1-\tau_{w_1}}+x^{p-1-(\tau_{w_1}+1)}+\ldots+1}\nonumber \\
		& & -
		\dfrac{1}{x^{p-1}+x^{p-1-\tau_{u_1}}}-\dfrac{1}{x^{p-1}+x^{p-1-\tau_{u_2}}}>0,
	\end{eqnarray}	
	if and only if its numerator is positive. The numerator is  
	\begin{eqnarray*}
		&&2x^{2(p-1)}(1+x^{-\tau_{u_1}})(1+x^{-\tau_{u_2}})\\
		&&-x^{p-1}(x^{p-1}+x^{p-1-\tau_{w_1}}+\ldots+1)(2+x^{-\tau_{u_1}}+x^{-\tau_{u_2}})\\
		&>& x^{2(p-1)}\left(2(1+x^{-\tau_{u_1}})(1+x^{-\tau_{u_2}})-\left(1+\dfrac{x^{-(\tau_{w_1-1})}}{x-1}\right)\left(2+x^{-\tau_{u_1}}+x^{-\tau_{u_2}}\right)\right)\\
		&\ge& 
		x^{2(p-1)}x^{-\tau_{u_1}}\left(1+x^{-(\tau_{u_2}-\tau_{u_1})}+2x^{-\tau_{u_2}}-\dfrac{2+x^{-\tau_{u_1}}+x^{-\tau_{u_2}}}{x-1}\right),
	\end{eqnarray*}
	since $\tau_{u_1}\le\tau_{w_1}-1$. 
	Thus~\eqref{ineq:1} holds true if 
	\[
	(x-1)(1+x^{-(\tau_{u_2}-\tau_{u_1})}+2x^{-\tau_{u_2}})\ge (2+x^{-\tau_{u_1}}+x^{-\tau_{u_2}})
	\] 
	and this is true for all $x\ge 4$.\\
	
	\noindent \textit{Case 2:} $\tau_{u_2}=p$. In this case $(u_2,u_2)_x=x^{p-1}$. Repeating the same calculations as in Case 1, we obtain
	\begin{eqnarray*}
		r_{\G_2}(x)-	r_{\G_1}(x)
		&\ge& 
		\dfrac{2}{x^{p-1}+x^{p-1-\tau_{w_1}}+x^{p-1-(\tau_{w_1}+1)}+\ldots+1} \\
		& & -
		\dfrac{1}{x^{p-1}+x^{p-1-\tau_{u_1}}}-\dfrac{1}{x^{p-1}}>0,
	\end{eqnarray*}
	if and only if
	\begin{eqnarray*}
		&& x^{2(p-1)}\left(2(1+x^{-\tau_{u_1}})-\left(1+\dfrac{x^{-(\tau_{w_1}-1)}}{x-1}\right)(2+x^{-\tau_{u_1}})\right)\\
		&=&x^{2(p-1)}x^{-\tau_{u_1}}\left(1-\dfrac{2+x^{-\tau_{u_1}}}{x-1} \right)\ge 0,
	\end{eqnarray*}
	since $\tau_{u_1}\le\tau_{w_1}-1$. The last inequality holds true if
	$x-1\ge 2+x^{-\tau_{u_1}}$, and this inequality is satified for all $x\ge 4$. \\
	
	\noindent \textit{Case 3:} $\tau_{w_1}=p$. In this case we have $(w_1,w_1)_x=x^{p-1}$. Note that
	\begin{equation*}
		r_{\G_2}(x)-	r_{\G_1}(x)
		\ge
		\dfrac{1}{x^{p-1}}\left(2 -
		\dfrac{1}{1+x^{-\tau_{u_1}}}-\dfrac{1}{1+x^{-\tau_{u_2}}}\right)>0,
	\end{equation*}
	if $x^{-\tau_{u_1}}+x^{-\tau_{u_2}}+2x^{-(\tau_{u_1}+\tau_{u_2})}>0$, which is true for all $x>0$.
\end{proof}

\begin{rems}\label{rem:cross-nonzero}
	1) Two collections with the same minimal period can give two different escape rates. For instance, consider the collections $\G_1=\{u_1=aaaa,u_2=bbbb\}$ and $\G_2=\{w_1=aaaa,w_2=bcbc\}$. Here $(u_1,w_1)_z=(w_1,u_1)_z=(u_2,w_2)_z=(w_2,u_2)_z=0$, $\tau_{\G_1}=\tau_{\G_2}=1$, but $\rho(\G_1)\ne \rho(\G_2)$.\\
	2) Theorem~\ref{thm:esc_min1} fails to hold when a cross-correlation polynomial between the words is non-zero. For instance, for the collections $\G_1=\{u_1=abc,u_2=bcd\}$ and $\G_2=\{w_1=abc,w_2=ddd\}$, $(u_1,w_1)_z\ne 0$, $3=\tau_{\G_1}=3>\tau_{\G_2}=1$, but $\rho(\G_1)<\rho(\G_2)$. 
\end{rems}

\begin{rem}[small values of $q$ not covered in Proposition~\ref{thm:esc_r_x} and Theorem~\ref{thm:esc_r_2}] \label{rem:lower_q}
	In Table~\ref{table:2}, $p=3$; observe that $r_{\G_i}(x)<r_{\G_{i+1}}(x)$, for all $x>4$ and hence $\rho(\G_i)<\rho(\G_{i+1})$ for all $q\ge 5$, from Proposition~\ref{thm:esc_r_x}. 
	
	Consider the collections $\G_1$ and $\G_{2}$ in Table~\ref{table:2}. Observe that $r_{\G_1}(x)<r_{\G_2}(x)$ for all $x> 4-4^{-1}$ and $\rho(\G_1)<\rho(\G_2)$ for $q=4$. Also, from Proposition~\ref{thm:esc_r_x}, $\rho(\G_1)<\rho(\G_2)$ for all $q\ge 5$. Further, consider the collections $\G_9$ and $\G_{10}$ in Table~\ref{table:2}. For some $x_0\in (3,4-4^{-1})$, $r_{\G_9}(x_0)=r_{\G_{10}}(x_0)$. Here the relation between $r_{\G_9}(x)$ and $r_{\G_{10}}(x)$ changes as $x$ goes from 3 to 4. In fact, $\rho(\G_{10})<\rho(\G_{9})$ for $q=3$, but $\rho(\G_{9})<\rho(\G_{10})$, for all $q\ge 4$. 
	
	A similar situation can be observed for the collections $\G_5$ and $\G_6$ in Table~\ref{table:1general}. Here $\rho(\G_6)<\rho(\G_5)$ for $q=2$, but $\rho(\G_5)<\rho(\G_6)$, for all $q\ge 3$. 
	
	Also, in Table~\ref{table:2}, observe that the rational functions corresponding to the collections $\G_1$ and $\G_2$ are different but the escape rates are the same, for $q=2$. 
	
	Hence for small values of $q$, not covered in Proposition~\ref{thm:esc_r_x} and Theorem~\ref{thm:esc_r_2}, nothing can be said about the relationship between the escape rate and rational function $r(z)$. 
	
	Similar observations can be made about the relationship between the escape rate and the minimal period of the hole. We give an example where the consequence in Theorem~\ref{thm:esc_min1} does not hold when $q=2$. Consider $\G_1=\{10111011,01001000\}$ and $\G_2=\{11100111,00011000\}$, which have zero cross-correlations. Here $\tau_{\G_1}=4<5=\tau_{\G_2}$, but $\rho(\G_1)>\rho(\G_2)$ at $q=2$. However, numerics suggest that Theorem~\ref{thm:esc_min1} holds true for all $p\ge 3$ and $q=3,4$.
\end{rem}

\section{Full shift with union of more than two cylinders as hole}\label{sec:full_morethantwo}
In this section, we extend the results obtained in the previous section to the case of full shift when the holes are unions of $t\ge 2$ cylinders (in other words, holes corresponding to $t\ge 2$ forbidden words). In Section~\ref{subsec:esc_r_1}, we explore the relationship between the escape rate and the corresponding rational function $r(z)$, and in Section~\ref{subsec:gen_cor_period}, we discuss the relationship between the escape rate and the minimal period of the hole. 

\subsection{Relationship between the escape rate and the rational function $r(z)$} \label{subsec:esc_r_1}
We first prove an easy lemma which is used in proving Theorem~\ref{thm:gen}.
\begin{lemma}\label{lemma:bound}
	Let $f,g$ be two polynomials of degree $m,n$, respectively, with $m\le n$. Let $M,N$ be the maxima in modulus of the coefficients of $f$ and $g$, respectively. Then the moduli of the coefficients of $fg$ are bounded above by $(m+1)MN$.  
\end{lemma}
\begin{proof}
	Let $f(x)=\sum_{k=0}^mf_kx^k$ and $g(x)=\sum_{k=0}^ng_kx^k$ with $m\le n$. 
	We have that $M=\max_k |f_k|$ and $N=\max_k|g_k|$. \\
	Since $(fg)(x)=\sum_{k=0}^{m+n}\left(\sum_{i+j=k}f_ig_j\right)x^k$, we will find an upper bound to  $|\sum_{i+j=k}f_ig_j|$. Since $m\le n$, $k$ can vary from $0$ to $m$ and thus there are at most $m+1$ terms in each sum. Therefore $|\sum_{i+j=k}f_ig_j|\le (m+1)MN$, and hence the result follows. 
\end{proof}
\begin{thm}\label{thm:gen}
	Suppose that $\mathcal{G}_1$ and $\mathcal{G}_2$ are finite collections of $t$ words each of length $p$ with symbols from $\Sigma$. Then for $p\ge 3$, there exists a positive constant $D(t,p)$ such that for all $q>D$, if $r_{\mathcal{G}_1}(x)<r_{\mathcal{G}_2}(x)$ at $x=D$, then $\rho(\mathcal{G}_1)<\rho(\mathcal{G}_2)$.
\end{thm}
\begin{proof}
	For a finite collection of $t$ words each of length $p$, let $\mathcal{M}(z)$ be the correlation matrix function. Let $r(z)$ be the sum of the entries of $\mathcal{M}^{-1}(z)$. Let $\D(z)$ be the determinant of $\mathcal{M}(z)$ and $\mathcal{S}(z)$ be the sum of the entries of the adjoint of $\mathcal{M}(z)$. Then $r(z)=\mathcal{S}(z)/\D(z)$. Also $\D(z)$ is a monic polynomial of degree $t(p-1)$ and $\mathcal{S}(z)$ is a polynomial of degree $(t-1)(p-1)$ with the leading coefficient $t$. Hence these polynomials can be expressed as 
	\[
	\D(z)=z^{t(p-1)}+\sum_{k=0}^{t(p-1)-1}a_kz^k,\ \ \mathcal{S}(z)=tz^{(t-1)(p-1)}+\sum_{k=0}^{(t-1)(p-1)-1}b_kz^k.
	\]
	If $a=\max_k\{|a_k|\}$ and $b=\max_k\{|b_k|,t\}$, then the proofs of Lemmas~\ref{Lemma:1} and~\ref{Lemma:2} can be imitated for obtaining the existence and position of the largest real root say $\lambda$, of   $(z-q)\D(z)+\mathcal{S}(z)$.
	
	First assume that $q>a+b+1$. Then on $|z|=a+b+1$, 
	\begin{equation*}
		\dfrac{|z-q||\D(z)|}{|\mathcal{S}(z)|}\ge \dfrac{|\D(z)|}{|\mathcal{S}(z)|}>|z|^{p-2}\dfrac{|z|-1-a}{b}= (a+b+1)^{p-2}>1,
	\end{equation*}
	for $p\ge 3$. Hence, as in Lemma~\ref{Lemma:1}, by Rouch\'e's theorem, $\lambda$ is the only zero of $(z-q)\D(z)+\mathcal{S}(z)$ on $|z|>a+b+1$. Since $|\D(z)|>0,$ $\lambda$ is the Perron root of the associated adjacency matrix. 
	
	Next, we prove that $q-1<\lambda<q$. Since
	\[
	\mathcal{S}(q)\ge tq^{(t-1)(p-1)}-\dfrac{bq^{(t-1)(p-1)-1}}{q-1}> q^{(t-1)(p-1)}\dfrac{t(q-1)-b}{q-1}>0, 
	\] 
	because $q>\dfrac{b}{t}+1$, we obtain $(q-x)\D(x)-\mathcal{S}(x)<0$, at $x=q$.
	
	Now at $x=q-1$,
	\begin{eqnarray*}
		(q-x)\D(x)-\mathcal{S}(x)\ge \dfrac{(q-1)^{(t-1)(p-1)+1}}{q-2}((q-2-a)(q-1)^{p-2}-b)>0,
	\end{eqnarray*}
	since $b\ge 2$, $p\ge 3$, and $q>a+b+1$ imply $(q-2-a)(q-1)^{p-2}-b>(b-1)(q-1)-b>(b-1)b-b\ge 0$.  
	
	Now we imitate the arguments in the proofs of Proposition~\ref{thm:esc_r_x} and Theorem~\ref{thm:esc_r_2}. By Lemma~\ref{lemma:bound}, the absolute values of the coefficients of $(\D_2\mathcal{S}_1-\D_1\mathcal{S}_2)$ are bounded above by $(a_1b_2+a_2b_1)((t-1)(p-1)+1)$, where for $i=1,2$, $a_i,b_i$ are the largest coefficients in modulus for $\D_i,\mathcal{S}_i$, respectively, as defined above. Thus, by Lagrange's bound on roots of a polynomial, all the zeroes of $r_{\G_1}-r_{\G_2}$ lie in the disc of radius $(a_1b_2+a_2b_1)((t-1)(p-1)+1)+1$. If we assume $D=D(t,p)=(a_1b_2+a_2b_1)tp+1$, then $r_{\G_1}(D)<r_{\G_2}(D)$ implies $r_{\G_1}(x)<r_{\G_2}(x)$, for all $x\ge D$. Hence for all $q>D$, $\rho(\G_1)<\rho(\G_2)$, the proof of which is similar to that of Proposition~\ref{thm:esc_r_x}.  
\end{proof}

\subsubsection{Upper bounds on $a$ and $b$ in terms of $p$ and $t$}
Let $\G=\{w_1,w_2,\ldots,w_t\}$ be a collection of words of length $p$.
The correlation matrix function is given by $\M(z)=[(w_j,w_i)_z]_{i,j}$, where $(w_i,w_i)_z$ has degree $p-1$ and $(w_i,w_j)_z$ has degree at most $p-2$, for $i\ne j$. The polynomial $\D(z)$ which is the determinant of $\M(z)$ can be written as 
\[
\D(z)=\sum_{\pi\in S_t}sgn(\pi)\left(\prod_{i=1}^t(w_{\pi i},w_i)_z\right),
\] where $S_t$ is the symmetric group on $t$ symbols. Choose $\pi\in S_t$. If $\pi$ fixes $j$ entries, then $\prod_{i=1}^t(w_{\pi i},w_i)_z$ is the product of $t$ polynomials, where $j$ of them have degree $p-1$, and the rest have degree at most $p-2$.  Also since each of these polynomials has coefficients either 0 or 1, the coefficients of $\prod_{i=1}^t(w_{\pi i},w_i)_z$ are bounded above by $(p-1)^{t-1}$, when $j=0$, and $p^{j-1}(p-1)^{t-j}$ otherwise (this is obtained by using inductive arguments in  Lemma~\ref{lemma:bound}). Let $k_t(j)$ denote the number of permutations in $S_t$ that fix exactly $j$ many symbols. Hence
\begin{equation}\label{eq:ubound}
	a\le k_t(0)(p-1)^{t-1}+\sum_{j=1}^t k_t(j)p^{j-1}(p-1)^{t-j}\le p^{t-1}\sum_{j=0}^t k_t(j)=p^{t-1}t!.
\end{equation}
For the purpose of finding an upper bound for $b$, we can assume that all the submatrices of $\mathcal{M}(z)$ of order $t-1$ are matrix functions where the diagonal entries have degree at most $p-1$ and the off-diagonal entries have degree at most $p-2$. Thus each entry of the adjoint matrix of $\M(z)$ is bounded above by $p^{t-2}(t-1)!$, using~\eqref{eq:ubound}. Hence 
\begin{equation}\label{eq:ubound2}
	b\le t^2p^{t-2}(t-1)!= t(t!)p^{t-2}.
\end{equation} 
Combining~\eqref{eq:ubound} and~\eqref{eq:ubound2}, we get 
$ab\le t(t!)^2p^{2t-3}$. 
Hence $D(t,p)$ can be chosen as $2t^2(t!)^2p^{2t-2}+1$. Using this method, for $t=2$, we obtain $D(2,p)=32p^2+1$, which is a much worse lower bound on $q$ than that obtained in Theorem~\ref{thm:esc_r_2}. 

\begin{rem}
	Theorem~\ref{thm:gen} may not hold true for small values of $q$ and $p$. One of the crucial steps in the proof is that the Perron root $\lambda$ is at least $q-1$. This is violated in the following example. Consider $q=3$, $\G=\{02,10,11,21,22\}$ ($p=2$), the Perron root is $\lambda=1.466$. 
\end{rem}

\subsection{Relationship between the escape rate and minimal period of the hole}\label{subsec:gen_cor_period}
In this section, we extend Theorem~\ref{thm:esc_min1} to the case when the hole corresponds to a union of $t\ge 2$ cylinders based at words of length $p\ge 2$. We obtain a lower bound on $q$ in terms of $t$. 

\begin{thm}\label{thm:gen-period}
	Suppose $t\ge 2$ and $q\ge 2$ satisfy 
	\begin{equation}\label{ineq:bound-q}
		(q-1)(t+1)-t q\left(1+\dfrac{1}{q}\right)^{t-1}\ge 0.
	\end{equation}	
	Let $\G_1=\{u_1,u_2,\ldots,u_t\}$ and $\G_2=\{w_1,w_2,\ldots,w_t\}$ be such that $u_i,w_i$ are words of equal length $p\ge 2$. Let $(u_i,u_j)_z=(w_i,w_j)_z=0$ for all $1\le i\neq j\le t$. If $\tau_{\G_1}<\tau_{\G_2}$, then $r_{\G_1}(q)<r_{\G_2}(q)$. 
	Inequality~\eqref{ineq:bound-q} holds true for $q\ge t 2^{t-1}+1$.
\end{thm}
\begin{proof} 
	Assume that $1\le\tau_{w_i},\tau_{u_i}<p$, $i=1,2,\ldots,t$. Let $(u_i,u_i)_q=q^{p-1}+q^{p-1-\tau_{u_i}}+R_{u_i}(q),(w_i,w_i)_q=q^{p-1}+q^{p-1-\tau_{w_i}}+R_{w_i}(q)$, where $R_{u_i},R_{w_i}$ are the reminder terms with degrees less than $p-1-\tau_{u_i},p-1-\tau_{w_i}$, respectively, for $i=1,\ldots,t$. Since all the cross-correlation polynomials between the words are zero, we have
	\[
	r_{\G_1}(q)=\dfrac{1}{(u_1,u_1)_q}+\dfrac{1}{(u_2,u_2)_q}+\ldots+\dfrac{1}{(u_t,u_t)_q},
	\]
	\[
	r_{\G_2}(q)=\dfrac{1}{(w_1,w_1)_q}+\dfrac{1}{(w_2,w_2)_q}+\ldots+\dfrac{1}{(w_t,w_t)_q}.
	\]
	Without loss of generality, we assume that $\tau_{u_1}\le\tau_{u_2}\le \ldots \le \tau_{u_t}$ and $\tau_{w_1}\le\tau_{w_2}\le \ldots \le \tau_{w_t}$. Furthermore, $\tau_{\G_1}<\tau_{\G_2}$ implies $\tau_{u_1}<\tau_{w_1}$. We consider the following three cases.\\ 
	
	Firstly, when $\tau_{w_1},\tau_{u_1},\ldots,\tau_{u_t}<p$, 
	\begin{eqnarray*}
		r_{\G_2}(q)-r_{\G_1}(q)&=&  \sum_{i=1}^t\dfrac{1}{(w_i,w_i)_q}- \sum_{i=1}^t\dfrac{1}{(u_i,u_i)_q}\\
		&\ge& \dfrac{t}{(w_1,w_1)_q}- \sum_{i=1}^t\dfrac{1}{q^{p-1}+q^{p-1-\tau_{u_i}}}\\
		&\ge& \dfrac{t}{q^{p-1}+\sum_{j=0}^{p-1-\tau_{w_1}}q^j}- \sum_{i=1}^t\dfrac{1}{q^{p-1}+q^{p-1-\tau_{u_i}}}.
	\end{eqnarray*}
	The last term in the above inequality is positive if and only if 
	\begin{eqnarray}\label{ineq:big}
		t \prod_{j=1}^t\left(1+q^{-\tau_{u_j}}\right)&>&
		\left(1+\dfrac{q^{-\tau_{u_1}}}{q-1}\right)\left(\sum_{i=1}^t\left( \prod_{j\ne i} (1+q^{-\tau_{u_j}}) \right)\right),
	\end{eqnarray}
	since $\tau_{u_1}\le\tau_{w_1}-1$. For $1\le j\le t$, let
	\[
	\beta_j=\sum_{\substack{1\le i_1\le \ldots\le i_j\le t}}q^{-\tau_{u_{i_1}}}\ldots q^{-\tau_{u_{i_j}}}.
	\]
	Then \[
	\prod_{j=1}^t\left(1+q^{-\tau_{u_j}}\right)=1+\beta_1+\ldots+\beta_t, \ \text{and}
	\]
	\[\sum_{i=1}^t\left( \prod_{j\ne i} (1+
	q^{-\tau_{u_j}}) \right)=t+(t-1)\beta_1+(t-2)\beta_2+\ldots+\beta_{t-1}.
	\]	
	\noindent Hence inequality\eqref{ineq:big} holds true if 
	\begin{eqnarray*}
		&& t\left(q-1\right)\left(1+\beta_1+\beta_2+\ldots+\beta_t\right) \\
		&& - \left(q-1+q^{-\tau_{u_1}}\right)\left(t+(t-1)\beta_1+(t-2)\beta_2+\ldots+\beta_{t-1}\right) \\
		&& =(q-1)(\beta_1+2\beta_2+\ldots+t\beta_{t})-q^{-\tau_{u_1}}(t+ (t-1)\beta_1+\ldots+\beta_{t-1})
		> 0.
	\end{eqnarray*}
	Since $1\le\tau_{u_1}\le\tau_{u_2}\le\ldots\le\tau_{u_t}$, we obtain
	\begin{eqnarray*}
		\beta_1&<&\beta_1q^{\tau_{u_1}}=1+q^{-(\tau_{u_2}-\tau_{u_1})}+\ldots+q^{-(\tau_{u_t}-\tau_{u_1})}\le t, \\
		\beta_2&<&\beta_2q^{\tau_{u_1}} = q^{-\tau_{u_2}}+\ldots+q^{-\tau_{u_t}}+q^{-(\tau_{u_2}+\tau_{u_3}-\tau_{u_1})}+\dots\\
		&& \dots +q^{-(\tau_{u_{t-1}}+\tau_{u_t}-\tau_{u_1})}
		\le  {t\choose 2}\dfrac{1}{q},\\
		&\vdots&\\
		\beta_k&<&\beta_kq^{\tau_{u_1}}\le {t\choose k} \dfrac{1}{q^{k-1}}.
	\end{eqnarray*}
	Since $\beta_1q^{\tau_{u_1}}>1,$ we have $(q-1)q^{\tau_{u_1}}(\beta_1+2\beta_2+\ldots+t\beta_{t})>q-1$.
	Hence
	\begin{eqnarray*}
		(q-1)q^{\tau_{u_1}}(\beta_1+2\beta_2+\ldots+t\beta_{t})-(t+ (t-1)\beta_1+\ldots+\beta_{t-1})\\
		>q-1-\left(t+\sum_{k=1}^{t-1}(t-k)\dfrac{1}{q^{k-1}}{t\choose k}\right).
	\end{eqnarray*}
	Now consider the term 
	\begin{equation*}
		\sum_{k=1}^{t-1}(t-k)\dfrac{1}{q^{k-1}}{t\choose k}=tq\sum_{k=1}^{t-1}{t-1\choose k}\dfrac{1}{q^{k}}
		=tq\left(\left(1+\dfrac{1}{q}\right)^{t-1}-1\right).
	\end{equation*} 
	Hence the inequality~\eqref{ineq:big} holds if~\eqref{ineq:bound-q} is satisfied. Also, the inequality~\eqref{ineq:bound-q} holds for $q\ge t 2^{t-1}+1$, since $\dfrac{1}{q^{k-1}}\le 1$, for any $1\le k\le t-1$.\\
	
	Secondly, when $\tau_{w_1}=p$, 
	\begin{eqnarray*}
		r_{\G_2}(q)-r_{\G_1}(q)&=&  \sum_{i=1}^t\dfrac{1}{(w_i,w_i)_q}- \sum_{i=1}^t\dfrac{1}{(u_i,u_i)_q}\\
		&\ge& \dfrac{t}{(w_1,w_1)_q}- \sum_{i=1}^t\dfrac{1}{q^{p-1}+q^{p-1-\tau_{u_i}}}\\
		&\ge& \dfrac{t}{q^{p-1}}- \sum_{i=1}^t\dfrac{1}{q^{p-1}+q^{p-1-\tau_{u_i}}}>0.
	\end{eqnarray*} \\
	
Finally, we consider the final case where $1\le \tau_{u_1},\tau_{u_2},\ldots,\tau_{u_\ell}<p$, and $\tau_{u_{\ell+1}}=\ldots=\tau_{u_t}=p$, for some $1\le \ell\le t$. With the same calculations as before, the result holds if 
	\[
	q-1-\left(t+\sum_{k=1}^{\ell-1}(t-k)\dfrac{1}{q^{k-1}}{t\choose k}\right)>q-1-\left(t+\sum_{k=1}^{t-1}(t-k)\dfrac{1}{q^{k-1}}{t\choose k}\right)\ge 0.
	\]	
\end{proof}

\begin{rem}
	Combining Theorems~\ref{thm:gen} and~\ref{thm:gen-period}, for $p\ge 3$ and $q\ge D$, if inequality~\eqref{ineq:bound-q} is satisfied at $q=D$ and if $\tau_{\G_1}<\tau_{\G_2}$, then $\rho(\G_1)<\rho(\G_2)$. 
\end{rem}

\section{Escape rate on a subshift of finite type}\label{sec:subshift}
In this section, we extend the results presented in the earlier sections for an irreducible subshift of finite type $\Sigma_{\F}$. We first consider the case when $\F=\{w\}$, and compare escape rates into two holes in $\Sigma_{\F}$ which correspond to the collections of forbidden words $\G_1=\{w_1\}$ and $\G_2=\{w_2\}$, with $\vert w\vert=\vert w_1\vert=\vert w_2\vert=p$. It follows from Theorems~\ref{thm:esc_r_1} and~\ref{thm:esc_r_2} that there exists a positive constant $D$, such that for $p=2$, $q\ge 2$, if $r_{\F\cup\G_1}(2)<r_{\F\cup\G_2}(2)$, and for $p\ge 3, q>3p^2+2$, if $r_{\F\cup\G_1}(3p^2+2)<r_{\F\cup\G_2}(3p^2+2)$, then $\rho(\F\cup\G_1)<\rho(\F\cup\G_2)$. Hence 
\begin{eqnarray*}
	\rho(\G_1;\Sigma_{\F})&=& h_{\text{top}}(\Sigma_{\F}) - h_{\text{top}}(\Sigma_{\F\cup \G_1})=\rho(\F\cup\G_1)-\rho(\F)\\
	&<&\rho(\F\cup\G_2)-\rho(\F)=	\rho(\G_2;\Sigma_{\F}).
\end{eqnarray*}
This can be seen using Tables~\ref{table:1s}, and ~\ref{table:2s}. Interestingly holes with same measure can have different escape rates and holes with different measures can have the same escape rate. For example, in Table~\ref{table:1s}, $q=3$, $\F=\{00\}$, $\G_1=\{w_1=11\}$, $\G_2=\{w_2=01\}$, and $\G_3=\{w_3=12\}$. Here $\mu(C_{w_1})=\mu(C_{w_3})$, but $\rho(\G_1;\Sigma_\F)\ne\rho(\G_3;\Sigma_\F)$. Moreover $\mu(C_{w_1})\ne\mu(C_{w_2})$, but $\rho(\G_1;\Sigma_\F)=\rho(\G_2;\Sigma_\F)$. Recall that $\mu$ denotes the Parry measure on $\Sigma_{\F}$.

\begin{table}[h] 
	\centering
	\caption{Escape rate values for $p=2$ and $q=3,4,\ldots,10$ for a subshift of finite type $\Sigma_\F$. Here $r_{\F\cup\G_i}(z)<r_{\F\cup\G_{i+1}}(z)$ for all real $z>2$.}
	\label{table:1s}
	\begin{footnotesize}
		\begin{tabular}{*3c|*5c} 
			\hline
			$\F$&$\G_1$&$\G_2$&$\F$&$\G_1$&$\G_2$&$\G_3$&$\G_4$\\
			
			$\{aa\}$&$\{ab\}/ \{ba\} /$&$\{bc\}$&$\{ab\}$&$\{aa\}/ \{bb\} /$&$\{ca\}/ \{bc\}$&$\{cc\}$&$\{cd\}/ \{ac\} /$   \\
			&$\{bb\}$&&&$\{ba\}$&&&  $\{cb\}$    \\\hline
			$q$ &$\rho(\G_1;\Sigma_\F)$	&$\rho(\G_2;\Sigma_\F)$&$q$ &$\rho(\G_1;\Sigma_\F)$	&$\rho(\G_2;\Sigma_\F)$ &$\rho(\G_3;\Sigma_\F)$	&$\rho(\G_4;\Sigma_\F)$	\\\hline
			3&0.1237&0.1955&3&0.0810&0.1188&0.1528&0.2693\\
			4&0.0625&0.0826&4&0.0468&0.0609&0.0668&0.0890\\
			5&0.0386&0.0475&5&0.0308&0.0378&0.0398&0.0491\\
			6&0.0264&0.0313&6&0.0220&0.0260&0.0269&0.0318\\
			7&0.0193&0.0222&7&0.0165&0.0191&0.0195&0.0225\\
			8&0.0147&0.0167&8&0.0129&0.0146&0.0149&0.0168\\ 
			9&0.0117&0.0130&9&0.0104&0.0116&0.0117&0.0131\\
			10&0.0095&0.0104&10&0.0085&0.0094&0.0095&0.0105\\\hline
		\end{tabular}
	\end{footnotesize}
\end{table}

\begin{table}[ht!] 
	\caption{Escape rate values for $p=2$ and $q=3,4,\ldots,10$ for a subshift of finite type $\Sigma_\F$. Here $r_{\F\cup\G_i}(z)<r_{\F\cup\G_{i+1}}(z)$ for all real $z>2$.}
	\centering
	\begin{footnotesize}
		\begin{tabular}{*6c} 
			\hline
			$\F=\{aaa\}$ & $\G_1=\{bbb\}/ \{aab\}/$ & $\G_2=\{aba\}$ & $\G_3=\{bcb\}$ & $\G_4=\{abb\}/ \{bba\} /$ & $\G_5=\{bcd\}$ \\
			& $\{baa\}$ & & & $\{abc\} / \{bca\}$ &  \\ \hline
			$q$ &$\rho(\G_1;\Sigma_\F)$	&$\rho(\G_2;\Sigma_\F)$ &$\rho(\G_3;\Sigma_\F)$	&$\rho(\G_4;\Sigma_\F)$	& $\rho(\G_5;\Sigma_\F)$  \\ \hline
			3&  0.0308 &  0.0345 &  0.0414 & 0.0429 & 0.0475\\
			4&  0.0129 &  0.0145 &  0.0162 & 0.0164& 0.0173\\
			5&  0.00675&  0.00756&  0.00809& 0.00815& 0.00843\\
			6&  0.00398&  0.00834&  0.00464& 0.00466&  0.00477\\
			7&  0.00255&  0.00281&  0.00291& 0.00292&   0.00297\\
			8&  0.00173&  0.00189&  0.00194& 0.001952&   0.001988\\
			9&  0.00123&  0.00134&  0.001368& 0.001369&  0.00138\\
			10&0.000907& 0.000979&  0.000997& 0.000998&   0.0010\\\hline
		\end{tabular}
	\end{footnotesize}
	\label{table:2s}
\end{table}

\noindent The next result is a straightforward consequence of the preceding discussion.

\begin{thm}
	Let $\F=\{w=a_1a_2\ldots a_p\}$ and $\mathcal{G}$ be the collection of all words $u$ of length $p$ such that $(w,u)_z=(u,w)_z=0$. Let $a,b\in \Sigma$ be distinct symbols in $\Sigma\setminus \{a_1,a_p\}$. Let $\G_u=\{u\}$, for all $u\in\mathcal{G}$. Then
	\[
	\min_{u\in\mathcal{G}}(\rho(\G_u;\Sigma_\F))=\rho(\G_{u_0};\Sigma_\F), \text{ and } \max_{u\in\mathcal{G}}(\rho(\G_u;\Sigma_\F))=\rho(\G_{u_1};\Sigma_\F),
	\]
	where $u_0=\bar{a}$, and $u_1=a\bar{b}$.
\end{thm}
\begin{proof}
	Let $r(z)$ be the rational function corresponding to the collection $\{w,u\}$, where $u\in\mathcal{G}$. Since $(w,u)_z=(u,w)_z=0$, we get $r(z)=\dfrac{1}{(w,w)_z}+\dfrac{1}{(u,u)_z}$. Note that $r(z)$ is minimum when $(u,u)_z=z^{p-1}+z^{p-2}+\ldots+1$, and maximum when $(u,u)_z=z^{p-1}$. Since $a,b\notin\{a_1,a_p\}$, $u_0,u_1\in\mathcal{G}$, $(u_0,u_0)_z=z^{p-1}+z^{p-2}+\ldots+1$, and $(u_1,u_1)_z=z^{p-1}$. 
\end{proof}

\begin{rem}
	Since the escape rate is invariant under conjugacy, all the results stated in this paper can be applied to maps that are conjugate to a subshift of finite type with a hole corresponding to a union of cylinders.
\end{rem}

The next result extends Theorem~\ref{thm:gen} when the underlying space is an irreducible subshift of finite type $\Sigma_\F$, for some collection $\F$ consisting of words of identical length with symbols from $\Sigma$.

\begin{thm}\label{thm:gen-subshift}
	Consider an irreducible subshift $\Sigma_{\F}$. Suppose $\G_1$ and $\G_2$ are finite collections consisting of $t$ allowed words in $\Sigma_\F$ each of length $p\ge 2$. Then for any $p\ge 3$, there exists a positive constant $D=D(t,p)$ such that for all $q> D$, if $r_{\F\cup\G_1}(x)<r_{\F\cup\G_2}(x)$ for $x=D$, then $\rho(\G_1;\Sigma_\F)<\rho(\G_2;\Sigma_\F)$.
\end{thm}

\begin{rem}
	Under the hypothesis of Theorem~\ref{thm:gen-period}, using Theorems~\ref{thm:gen-period} and~\ref{thm:gen-subshift}, for the collection of forbidden words $\F\cup\G_i$, $i=1,2$, for $p\ge 3$ and for $q>D$ satisfying inequality~\eqref{ineq:bound-q}, we obtain $\rho(\G_1;\Sigma_\F)<\rho(\G_2;\Sigma_\F)$.
\end{rem}

\section{Concluding remarks and future directions}\label{sec:conc}
Remark~\ref{rem:lower_q} highlights various situations for small values of $p$ and $q$ which are not covered in Theorems~\ref{thm:esc_r_2} and~\ref{thm:esc_min1}. However, numerical calculations suggest that the bound $q>3p^2+2$ is not the best. One of the reasons for this is that our proofs consider only the fact that the correlation polynomials have a certain degree and have coefficients 0 or 1. These polynomials have more structure to them. For example, if we consider two different words of same length of the types $aa\ldots a$ and $bb\ldots b$, then their autocorrelation polynomials have all their coefficients 1 but all their cross-correlations must be 0. In fact, the upper bound that we obtain for the coefficients of the rational function $r$ is never achieved by any collection of words. This issue will be taken up in our future work.

This study leads to several other interesting problems. As illustrated in Remark~\ref{rem:Markov}, if the underlying subshift is induced with a general Markov measure which is not the Parry measure, then for two holes with the same correlation function, the escape rate may not be the same. This raises the issue of the relationship between the escape rate and the topological entropy, in general. The question of how the escape rate changes with $q$ can also be explored. The numerics presented in the tables suggest that the escape rate decreases with increasing $q$, which is expected since the size of the hole reduces. If this is true in general, we would like to know at what rate the escape rate decays. The numerics presented in the tables also suggest that Theorem~\ref{thm:esc_r_2} holds true for all values of $p\geq 2$ and $q\geq 2$. Our proof is not helpful in this regard. Another interesting question is the dependence of the escape rate on the Poincar\'e recurrence time when the cross-correlation between the forbidden words is non-zero. It is certain from Remark~\ref{rem:cross-nonzero} that it does not depend on the minimal period of the hole. A question of general interest is to explore other factors that influence the escape rate into the hole other than the length, the number of corresponding forbidden words, and the minimum period of the hole. In~\cite{BB}, Bolding and Bunimovich discuss finite time dynamical properties of shift map on a full shift with cylinders as holes. For each $n\ge 1$, they compare the values of $f_w(n)$, the number of words of length $n$ which end with $w$ but do not contain $w$ as subwords in any other place. This issue can be studied in our setup -- subshift of finite type with holes.

\section{Funding}
The research of the first author is supported by the Council of Scientific \& Industrial Research (CSIR), India (File no.~09/1020(0133)/2018-EMR-I), and the second author is supported by Center for Research on Environment and Sustainable Technologies (CREST), IISER Bhopal, CoE funded by the Ministry of Human Resource Development (MHRD), India.

\addcontentsline{toc}{chapter}{References}
\bibliographystyle{plain}   
\renewcommand{\bibname}{References} 
\bibliography{mybib.bib} 

\end{document}